\documentclass{compositio}

\usepackage{amsfonts,latexsym,amsmath,amssymb,amsthm,amscd,euscript,upgreek}
\usepackage{mathrsfs,manfnt,enumerate}
\usepackage[cal=boondox,scr=boondoxo]{mathalfa}
\usepackage[super]{nth}

\usepackage{verbatim,enumitem}
\usepackage{calc,color,placeins}
\usepackage{amsxtra}
\usepackage[all]{xy}
\usepackage{graphicx}
\usepackage{hyperref}
\usepackage{mathtools}
\usepackage{pifont}
\usepackage{tikz-cd}
\usetikzlibrary{matrix,shapes,arrows,decorations.pathmorphing,calc}
\usetikzlibrary{decorations.pathreplacing,decorations.markings}
\usepackage[new]{old-arrows}
\usepackage{calculator}
\usepackage{scalerel,xcolor}
\def\darrow{\mathrel{\ThisStyle{\ooalign{$\SavedStyle\rightarrow$\cr%
  \hfil\textcolor{white}{\rule{2\LMpt}{1\LMex}}\kern2\LMpt\hfil}}}}

\newtheorem{theorem}{Theorem}
\newtheorem{proposition}[theorem]{Proposition}
\newtheorem{lemma}[theorem]{Lemma}

\theoremstyle{definition}

\theoremstyle{remark}
\newtheorem{remark}[theorem]{Remark}

\newcommand{\id}{\operatorname{id}}
\newcommand{\defeq}{\vcentcolon=}

\newcommand\nc{\newcommand}
\nc{\on}{\operatorname}
\nc\renc{\renewcommand}
\nc{\BR}{\mathbb R}
\nc{\BC}{\mathbb C}
\nc{\BQ}{\mathbb Q}
\nc{\BF}{\mathbb F}
\nc{\BZ}{\mathbb Z}
\nc{\BN}{\mathbb N}
\nc{\BS}{\mathbb S}
\nc{\BA}{\mathbb A}
\nc{\BP}{\mathbb P}
\nc{\Hom}{\on{Hom}}
\nc{\wt}{\widetilde}
\nc{\vspan}{\on{span}}
\nc{\ord}{\on{ord}}
\nc{\im}{\on{im}}
\nc{\Mat}{\on{Mat}}
\nc{\can}{\on{can}}
\nc{\coker}{\on{coker}}
\nc{\ev}{\on{ev}}
\nc{\Tr}{\on{Tr}}
\nc{\End}{\on{End}}
\nc{\Aut}{\on{Aut}}
\nc{\swap}{\on{swap}}
\nc{\Set}{\on{Set}}
\nc{\bC}{{\mathbf C}}
\nc{\bc}{{\mathbf c}}
\nc{\bD}{{\mathbf D}}
\nc{\bd}{{\mathbf d}}
\nc{\bE}{{\mathbf E}}
\nc{\be}{{\mathbf e}}
\nc{\bF}{{\mathbf F}}
\nc{\bff}{{\mathbf f}}
\nc{\fa}{\mathfrak a}
\renc{\mod}{\on{-mod}} 

\nc{\adj}{\on{adj}}
\nc{\tensor}[3]{#1 \underset{#2}\otimes #3}
\nc{\Nat}{\on{Nat}}
\nc{\op}{\on{op}}
\nc{\Funct}{\on{Funct}}
\nc{\Ob}{\on{Ob}}
\nc{\fR}{\mathfrak{R}}
\nc{\Vect}{\on{Vect}}
\nc{\ns}{\on{non-spec}}
\nc{\GL}{\on{GL}}
\nc{\ol}{\overline}
\nc{\ul}{\underline}
\nc{\univ}{\on{univ}}
\nc{\Maps}{\on{Maps}}
\nc{\bdd}{\on{bdd}}
\nc{\cont}{\on{cont}}
\nc{\Sym}{\on{Sym}}
\nc{\Ind}{\on{Ind}}
\nc{\Res}{\on{Res}}
\nc{\Ann}{\on{Ann}}
\nc{\cI}{\mathcal{I}}
\nc{\pt}{\on{pt}}
\nc{\Bl}{\on{\Bl}}
\nc{\Spec}{\on{Spec}}
\nc{\Cl}{\on{Cl}}
\nc{\hannahbox}{\boxed}
\nc{\cD}{\mathcal{D}}
\renc{\div}{\on{div}}
\nc{\mc}{\mathcal}
\nc{\pp}{\mathfrak{p}}
\nc{\scr}{\mathscr}

\makeatletter
\newcommand{\extp}{\@ifnextchar^\@extp{\@extp^{\,}}}
\def\@extp^#1{\mathop{\bigwedge\nolimits^{\!#1}}}
\makeatother

\tikzset{%
    add/.style args={#1 and #2}{
        to path={%
 ($(\tikztostart)!-#1!(\tikztotarget)$)--($(\tikztotarget)!-#2!(\tikztostart)$)%
  \tikztonodes},add/.default={.2 and .2}}
}

\tikzset{
  on each segment/.style={
    decorate,
    decoration={
      show path construction,
      moveto code={},
      lineto code={
        \path [#1]
        (\tikzinputsegmentfirst) -- (\tikzinputsegmentlast);
      },
      curveto code={
        \path [#1] (\tikzinputsegmentfirst)
        .. controls
        (\tikzinputsegmentsupporta) and (\tikzinputsegmentsupportb)
        ..
        (\tikzinputsegmentlast);
      },
      closepath code={
        \path [#1]
        (\tikzinputsegmentfirst) -- (\tikzinputsegmentlast);
      },
    },
  },
  rightend arrow/.style={postaction={decorate,decoration={
        markings,
        mark=at position .81 with {\arrow[#1]{triangle 45}}
      }}},
      leftend arrow/.style={postaction={decorate,decoration={
        markings,
        mark=at position .18 with {\arrowreversed[#1]{triangle 45}}
      }}},
}

\makeatletter
\@namedef{subjclassname@2020}{%
  \textup{2020} Mathematics Subject Classification}
\makeatother

\usepackage[draft]{todonotes}

\allowdisplaybreaks

\title[Most Odd-Degree Binary Forms Fail to Properly Represent a Square]{Most Odd-Degree Binary Forms Fail to \\ Primitively Represent a Square}

\author{Ashvin A.~Swaminathan}
\email{ashvins@alumni.princeton.edu}
\address{Department of Mathematics, Harvard University, \mbox{Cambridge, MA 02138}, USA}

\thanks{This research was supported by the Paul and Daisy Soros Fellowship and the NSF Graduate Research Fellowship.}


\date{\today}

\begin{document}

\begin{abstract}
    \noindent Let $F$ be a separable integral binary form of odd degree $N \geq 5$. A result of Darmon and Granville known as ``Faltings plus epsilon'' implies that the degree-$N$ \emph{superelliptic equation} $y^2 = F(x,z)$ has finitely many primitive integer solutions. In this paper, we consider the family $\mathscr{F}_N(f_0)$ of degree-$N$ superelliptic equations with fixed leading coefficient $f_0 \in \mathbb{Z} \smallsetminus \pm\mathbb{Z}^2$, ordered by height. For every sufficiently large $N$, we prove that among equations in the family $\mathscr{F}_N(f_0)$, more than $74.9\%$ are insoluble, and more than $71.8\%$ are everywhere locally soluble but fail the Hasse principle due to the Brauer--Manin obstruction. We further show that these proportions rise to at least $99.9\%$ and $96.7\%$, respectively, when $f_0$ has sufficiently many prime divisors of odd multiplicity.
    Our result can be viewed as a strong asymptotic form of ``Faltings plus epsilon'' for superelliptic equations and constitutes an analogue of Bhargava's result that most hyperelliptic curves over $\mathbb{Q}$ have no rational points.
\end{abstract}

\maketitle


\section{Introduction} \label{sec-intro}

Let $F$ be a separable integral binary form of degree $N \geq 5$. When $N = 2n$ is even, the subscheme $C_F$ of the weighted projective plane $\BP_\BQ^2(1,n,1)$ cut out by the equation $y^2 = F(x,z)$ is known as a \emph{hyperelliptic curve}, and by Faltings' Theorem (see~\cite[Satz~7]{MR718935}), the curve $C_F$ has only finitely many rational points. In~\cite[Theorem~1]{thesource}, Bhargava established the following ``strong asymptotic form'' of Faltings' Theorem for hyperelliptic curves: when integral binary $N$-ic forms are ordered by height, the density of forms $F$ such that the curve $C_F$ has a rational point is $o(2^{-n})$. In other words, most even-degree binary forms fail to represent a square.

 When $N= 2n+1$ is odd, the equation $y^2 = F(x,z)$ is called \emph{superelliptic}. In contrast with the case of hyperelliptic curves, the problem of studying rational solutions to superelliptic equations is trivial: given any pair $(x_0,z_0) \in \BQ^2$, the triple
\begin{equation} \label{eq-triv}
(x,y,z) = (x_0 \cdot F(x_0,z_0), F(x_0,z_0)^{n+1}, z_0 \cdot F(x_0, z_0)) \in \BQ^3
\end{equation}
is a solution to the equation $y^2 = F(x,z)$, and in fact, the subscheme $S_F \subset \BP_\BQ^2(2,2n+1,2)$ cut out by this equation is isomorphic to $\mathbb{P}_{\BQ}^1$. On the other hand, the problem of studying \emph{primitive integer solutions} --- triples $(x_0,y_0,z_0) \in \mathbb{Z}^3$ such that $y_0^2 = F(x_0,z_0)$ and $\gcd(x_0,z_0) = 1$ --- is considerably more interesting. For instance, a result of Darmon and Granville states that $y^2 = F(x,z)$ has only finitely many primitive integer solutions (see~\cite[Theorem~1']{MR1348707}). 
Given this analogue of Faltings' Theorem, it is natural to expect that an analogue of Bhargava's strong asymptotic form holds for primitive integer solutions to superelliptic equations. The purpose of this paper is to prove this expectation, i.e., that most integral odd-degree binary forms fail to \emph{primitively} represent a square.

\subsection{Superelliptic Stacky Curves} In the rest of the paper, it will be convenient to reinterpret primitive integer solutions to the superelliptic equation $y^2 = F(x,z)$ as integral points on a certain \emph{stacky curve}, which we denote by $\mathscr{S}_F$ and define as follows. Consider the punctured affine surface
\begin{equation} \label{eq-surfacedef}
\wt{S}_F \defeq V(y^2 = F(x,z)) \subset \BA^3_\BZ \smallsetminus \{(0,0,0)\}.
\end{equation}
The group $\mathbb{G}_m$ acts by scaling on $\wt{S}_F$ via $\lambda \cdot (x_0,y_0,z_0) = (\lambda^2 \cdot x_0,\lambda^{2n+1} \cdot y_0, \lambda^2 \cdot z_0)$. The map $\wt{S}_F \to \BP_\BZ^1$ defined by $(x,y,z) \to [x : z]$ is $\mathbb{G}_m$-equivariant, and the field of $\mathbb{G}_m$-invariant rational functions on $\wt{S}_F$ is generated by $x/z$, so the scheme quotient $\wt{S}_F/\mathbb{G}_m$ is isomorphic to $\mathbb{P}_\BZ^1$. \mbox{The associated} stack quotient $\mathscr{S}_F \defeq [\wt{S}_F / \mathbb{G}_m]$ is a stacky curve with coarse moduli space $\mathbb{P}^1_\mathbb{Z}$. The $\mathbb{G}_m$-action on $\wt{S}_F$ has a nontrivial stabilizer isomorphic to the group scheme $\mu_2$ if and only if $y = 0$, so $\mathscr{S}_F$ has a ``$\tfrac{1}{2}$-point'' at each of the $2n+1$ distinct roots of $F$ over the algebraic closure.

Let $R$ be a principal ideal domain. We say that a triple $(x_0, y_0, z_0) \in R^3$ is an $R$\emph{-primitive solution} to $y^2 = F(x,z)$ if $y_0^2 = F(x_0,z_0)$ and $Rx_0 + Rz_0 = R$. The set of $R$-primitive solutions to $y^2 = F(x,z)$ is readily seen to be in bijection with the set $\wt{S}_F(R)$ of morphisms $\Spec R \to \wt{S}_F$ (i.e., $R$-points of $\wt{S}_F$), which in turn corresponds bijectively with the set $\mathscr{S}_F(R)$: indeed, specifying a map $\Spec R \to \mathscr{S}_F$ is definitionally equivalent to specifying a map $L \to \wt{S}_F$, where $L$ is a torsor of $\Spec R$ by $\mathbb{G}_m$, but $\Spec R$ has no nontrivial $\mathbb{G}_m$-torsors. In what follows, we abuse notation by writing $\mathscr{S}_F(R)$ for the set of $R$-primitive solutions to $y^2 = F(x,z)$, regardless of whether $F$ is separable.

\subsection{Main Results}
Our results concern families of superelliptic equations $y^2 = F(x,z)$ of degree $2n+1 \geq 3$ having fixed leading coefficient $f_0 \defeq F(1,0) \in \BZ \smallsetminus \{0\}$. Suppose that $|f_0|$ is not a square, let $|f_0| = m^2\upkappa$, where $m,\upkappa \in \BZ$ and $\upkappa$ is squarefree, and define the following quantities:
$$\mu_{f_0} \defeq \prod_{p \mid \upkappa}\left(\frac{1}{p^2} + \frac{p-1}{p^{n+1}}\right) \quad \text{and} \quad \mu_{f_0}' \defeq \prod_{p \mid \upkappa} \left(1 - \frac{1}{p^{2p+1}}\right)$$
For a $\BZ$-algebra $R$, we write $\mathscr{F}_{2n+1}(f_0, R) \subset R[x,z]$ for the set of binary forms of degree \mbox{$2n+1$} with leading coefficient equal to $f_0$; when $R = \BZ$, we abbreviate $\mathscr{F}_{2n+1}(f_0,\BZ)$ as $\mathscr{F}_{2n+1}(f_0)$. For simplicity, we sometimes write $\mathscr{F}_{2n+1}(f_0)$ to mean the family of superelliptic equations or superelliptic stacky curves defined by the forms in $\mathscr{F}_{2n+1}(f_0)$. The \emph{height} $H(F)$ of a binary form $F(x,z) = \sum_{i = 0}^{2n+1} f_i x^{2n+1-i}z^i \in \mathscr{F}_{2n+1}(f_0)$ is defined as follows
$$H(F) \defeq \max\{|f_0^{i-1}f_i|^{2n(2n+1)/i} : i = {1, \dots, 2n+1}\}.$$
With this setup, we prove the following analogue for superelliptic stacky curves of~\cite[Theorem~1]{thesource}:
\begin{theorem} \label{thm-main}
Let $f_0$ be a non-square integer. For every sufficiently large $n$, most superelliptic stacky curves in the family $\mathscr{F}_{2n+1}(f_0)$ have no primitive integer solutions. More precisely:
\begin{enumerate}[leftmargin=2em]
    \item[$(\mathrm{a})$] Suppose $2 \nmid f_0$. For every $n \geq 5$, a positive proportion of forms $F \in \mathscr{F}_{2n+1}(f_0)$, when ordered by height, have the property that $\mathscr{S}_F(\BZ) = \varnothing$. Moreover, the lower density of forms $F \in \mathscr{F}_{2n+1}(f_0)$ such that $\mathscr{S}_F(\BZ) = \varnothing$ is at least $1 - \mu_{f_0} + o(2^{-n})$.
    \item[$(\mathrm{b})$] Suppose $2 \mid f_0$. The lower density of forms $F \in \mathscr{F}_{2n+1}(f_0)$ such that $\mathscr{S}_F(\BZ) = \varnothing$ is at least $1 - \mu_{f_0} + O(2^{-\varepsilon_1n^{\varepsilon_2}})$ for some real numbers $\varepsilon_1,\varepsilon_2 > 0$.
\end{enumerate}
\end{theorem}

As shown in Table~\ref{tab-display}, the quantity $1-\mu_{f_0}$ --- which by Theorem~\ref{thm-main} constitutes a lower bound on the density of insoluble equations in $\mathscr{F}_{2n+1}(f_0)$ in the large-$n$ limit --- rapidly approaches $1$ as the number of prime factors of $\upkappa$ grows. Furthermore, the method used to prove Theorem~\ref{thm-main} can be easily adapted to obtain a similar result when the family $\mathscr{F}_{2n+1}(f_0)$ is replaced by any subfamily defined by congruence conditions modulo finitely many prime powers.

Note that Theorem~\ref{thm-main} is subject to the condition that $|f_0|$ is \emph{not} a square. When $|f_0|$ is a square, every form $F \in \mathscr{F}_{2n+1}(f_0)$ is such that $y^2 = F(x,z)$ has a trivial solution, namely $(x,y,z) = (\on{sign}(f_0),\sqrt{|f_0|},0)$. We expect that for most $F \in \mathscr{F}_{2n+1}(f_0)$, the equation $y^2 = F(x,z)$ has no solutions other than the trivial one. The presence of the trivial solution renders this expectation difficult to prove: e.g., to obtain the analogous result for rational points on monic odd-degree hyperelliptic curves, Poonen and Stoll~\cite{MR3245014} combined work of Bhargava and Gross~\cite{MR3156850} on equidistribution of the $2$-Selmer elements of the Jacobians of these curves with a variant of Chabauty's method (cf.~\cite{MR3719247}, in which Shankar and Wang treat the even-degree case, as well as the work of Thorne and Romano~\cite{MR4258170} and Laga~\cite{laga1}, which prove similar theorems for certain families of non-hyperelliptic curves). To the author's knowledge, analogues of the notion of Selmer group and of Chabauty's method remain to be formulated for superelliptic stacky curves.

\FloatBarrier
\begin{table}
    \centering
     \caption{A table demonstrating the strength of the limiting lower bounds $\lim_{n \to \infty}1 - \mu_{f_0}$ (resp., $\lim_{n \to \infty} \mu_{f_0}' - \mu_{f_0}$) on the density of insoluble superelliptic equations given by Theorem~\ref{thm-main} (resp., on the density of superelliptic equations having a Brauer--Manin obstruction to solubility given by Theorem~\ref{thm-hasse}).}
    \begin{tabular}{ |c|c|c|c|c|c|c|c| }
    \hline
 $\{p \mid \upkappa\} \supset$ & $\{2\}$ & $\{3\}$ & $\{5\}$ & $\{2,3\}$ & $\{2,5\}$ &  $\{3, 7\}$ & $\{2, 3, 7\}$ \\ \hline
 $\lim_{n \to \infty}1 - \mu_{f_0} \geq $ & 75.0\% & 88.8\% & 96.0\% & 97.2\% & 99.0\% & 99.7\% & 99.9\% \\
 \hline
  $\lim_{n \to \infty}\mu_{f_0}' - \mu_{f_0} \geq $ & 71.8\% & 88.8\% & 95.9\% & 94.0\% & 95.8\% & 99.7\% & 96.7\% \\
 \hline
\end{tabular}
\vspace*{0.5cm}
\label{tab-display}

\vspace*{-17pt}
\end{table}

In the course of proving Theorem~\ref{thm-main}, we actually prove the stronger statement that most superelliptic stacky curves have a $2$-descent obstruction to having an integral point. Since these stacky curves typically have integral points everywhere locally, it follows that most of them fail the Hasse principle (i.e., do \emph{not} have a $\BZ$-point despite having a $\BZ_v$-point for every place $v$ of $\BQ$) due to a $2$-descent obstruction. Upon showing that such a $2$-descent obstruction yields a case of the \mbox{Brauer--Manin obstruction,} we obtain the following analogue for superelliptic stacky curves of~\cite[Corollary~4]{thesource}\footnote{The proof of~\cite[Corollary~4]{thesource} has a gap; see Remark~\ref{rmk-thefix}, where we explain how a certain step in the proof of Theorem~\ref{thm-hasse} may be modified to obtain a complete proof of~\cite[Corollary~4]{thesource}.} (see also the work of Thorne~\cite{MR3298319} and Thorne and Romano~\cite{MR3787911}, which prove similar theorems for certain families of non-hyperelliptic curves):
\begin{theorem} \label{thm-hasse}
Let $f_0$ be a non-square integer. The lower density of forms $F \in \mathscr{F}_{2n+1}(f_0)$, when ordered by height, such that the stacky curve $\mathscr{S}_F$ fails the Hasse principle due to the Brauer--Manin obstruction is at least \mbox{$\mu_{f_0}' - \mu_{f_0} + o(2^{-n})$} if $2 \nmid f_0$ and is at least $\mu_{f_0}' - \mu_{f_0} + O(2^{-\varepsilon_1n^{\varepsilon_2}})$ for some $\varepsilon_1,\varepsilon_2 > 0$ if $2 \mid f_0$.
\end{theorem}

 The proportion $\mu_{f_0}'-\mu_{f_0}$, which can be thought of as a lower bound on the density of superelliptic equations with leading coefficient $f_0$ having a Brauer--Manin obstruction to solubility in the large-$n$ limit, approaches $\approx96.7\%$ as the number of prime factors of $\upkappa$ grows; see Table~\ref{tab-display}.

 The \emph{genus} of a stacky curve can be defined in terms of the Euler characteristic of an orbifold curve; see~\cite[section entitled ``Orbifolds, and the topology of $M$-curves'']{MR1479291}. In~\cite{BPpreprint}, Bhargava and Poonen prove that all (suitably defined) stacky curves $S$ of genus less than $\frac{1}{2}$ over $\BZ$ satisfy the Hasse principle; i.e., if $S(\BR) \neq \varnothing$ and $S(\BZ_p) \neq \varnothing$ for every prime $p \in \BZ$, then $S(\BZ) \neq \varnothing$. On the other hand, there is no guarantee that a stacky curve of genus at least $\frac{1}{2}$ satisfies the Hasse principle: indeed, for any separable integral binary form $F$ of degree $2n+1 \geq 3$, the genus of the associated stacky curve $\mathscr{S}_F$ is given by $\frac{2n+1}{4} > \frac{1}{2}$, and Theorem~\ref{thm-hasse} demonstrates that these stacky curves often fail the Hasse principle.  

\subsection{Method of Proof} \label{sec-method}

Our strategy is to transform the problem of counting soluble superelliptic equations into one of counting integral orbits of a certain representation by devising a suitable orbit construction. In \S\ref{sec-buildabear}, we prove that elements of $\mathscr{S}_F(\BZ)$ naturally give rise to certain square roots of the ideal class of the inverse different of a ring $R_F$ associated to the \mbox{form $F$.} Via a parametrization that we introduced in~\cite{Swpreprint}, these square roots correspond to certain integral orbits for the action of the split odd special orthogonal group $G$ on a space $V$ of self-adjoint operators $T$, where the characteristic polynomial of $T$ is equal to the \emph{monicized form} $F_{\on{mon}}(x,1) \defeq \frac{1}{f_0} \cdot F(x,f_0z)$.

The parametrization in~\cite{Swpreprint} is explicitly leading-coefficient-dependent and is thus particularly well-suited for applications concerning binary forms with fixed leading coefficient. By generalizing an equidistribution argument that we developed in joint work with Bhargava and Shankar (see~\cite{BSSpreprint}, where we used the very same parametrization to determine the second moment of the size of the $2$-Selmer group of elliptic curves), one can average the bounds in Theorems~\ref{thm-main} and~\ref{thm-hasse} over all leading coefficients and thus deduce analogues of these theorems for the full family of superelliptic equations of given degree. In this paper, we fix the leading coefficient for two reasons: doing so simplifies the orbit-counting and allows us to prove a stronger result, that within each \emph{thin} subfamily of superelliptic equations with fixed leading coefficient, most members fail the Hasse principle.

  In \S\ref{sec-algprems}, we discuss the representation of \mbox{$G$ on $V$,} the orbits of which also correspond to $2$-Selmer elements of the Jacobians of monic odd-degree hyperelliptic curves (see the work of Bhargava and Gross~\cite{MR3156850}), we define the ring $R_F$, and we recall the parametrization from~\cite{Swpreprint}. In \S\ref{sec-finally}, we prove Theorem~\ref{thm-main} in a series of three steps. First, in \S\ref{sec-quotebharg}, we adapt results of Bhargava and Gross (see~\cite[\S10]{MR3156850}) to obtain a bound on the number of $G(\BZ)$-orbits on $V(\BZ)$ satisfying local conditions. Second, in \S\ref{sec-overr}-\ref{sec-overzpqp}, we bound the local volume of the set of elements in $V(\BZ_v)$ that arise from $\BZ_v$-primitive solutions via the orbit construction in \S\ref{sec-buildabear} for each place $v$ of $\BQ$. This is the most technical step of the proof, and we handle it in a series of subsections as follows:
  \begin{itemize}
      \item In \S\ref{sec-overr}, we take $v = \infty$ and show that the relevant local volume contributes a factor of $O(2^n)$ to the final bound.
      \item In \S\ref{sec-overp}, we take $v = p \nmid 2f_0$ and obtain a bound on the relevant local volume; later, in \S\ref{sec-hillary}, we show that the product of this bound over all $p \nmid 2f_0$ contributes a factor of $O(2^{-\varepsilon_1 n^{\varepsilon_2}})$ to the final bound; this constitutes part of the saving when $2 \nmid f_0$ and all of the \mbox{saving when $2 \mid f_0$.}
      \item In \S\ref{sec-orbates}, we take $v = 2 \nmid f_0$ and show that the relevant local volume contributes a factor of $O(2^{-2n})$ to the final bound; this constitutes most of the saving when $2 \nmid f_0$.
      \item In \S\ref{sec-overzpqp}, we take $v = p \mid f_0$ and show that the relevant local volume contributes a factor of $O(1)$ to the final bound.
  \end{itemize}
  The main idea for the savings obtained in \S\S\ref{sec-overp}-\ref{sec-orbates} is that the number of solutions to $y^2 = F(x,z)$ modulo an integer $k$ is at most $\#\mathbb{P}^1(\BZ/k\BZ)$, which, for $k$ small relative to the degree $n$, is typically smaller than the total number of $G(\BZ/k\BZ)$-orbits having characteristic polynomial $F_{\mathrm{mon}}$ --- for example, when $k = p \nmid f_0$ is an odd prime, this number of orbits is $\asymp 2^m$, where $m$ is the number of distinct irreducible factors of $F_{\mathrm{mon}}$ modulo $p$. Third, in \S\ref{sec-hillary}, we combine the results of the first two steps to complete the proof of Theorem~\ref{thm-main}. Finally, in \S\ref{sec-bm}, we define the notions of $2$-descent obstruction and Brauer--Manin obstruction for superelliptic stacky curves, and we prove Theorem~\ref{thm-hasse}.

\section{Algebraic Preliminaries} \label{sec-algprems}

In this section, we recall several algebraic notions that feature in our construction of orbits from integral points on superelliptic stacky curves. We begin in \S\ref{sec-rep} by introducing the representation of $G$ on $V$. Then, in \S\ref{sec-ringsbins}, we define the ring $R_F$ associated to a binary form $F$, and in \S\ref{sec-param}, we recall a parametrization from~\cite{Swpreprint} of square roots of the ideal class of the inverse different of $R_F$. We finish in \S\ref{sec-fieldgens} by describing the parametrization over fields.

\subsection{A Representation of the Split Odd Special Orthogonal Group} \label{sec-rep}

Let $W$ be the $\BZ$-lattice of rank $2n+1$ equipped with a nondegenerate symmetric bilinear form \mbox{$[-,-] \colon W \times W \to \BZ$} that has signature $(n+1,n)$ after extending scalars to $\BR$. By~\cite[Chapter V]{MR0344216}), the lattice $W$ is unique up to isomorphism over $\BZ$, and there is a $\BZ$-basis
\begin{equation} \label{eq-Wbasis}
    W = \BZ\langle e_1, \dots, e_n,u, e_1', \dots, e_n' \rangle
\end{equation}
such that $[e_i, e_j] = [e_i',e_j'] = [e_i,u]= [e_i',u]= 0$, $[e_i, e_j'] = \delta_{ij}$, and $[u,u] = 1$ for all relevant pairs $(i,j)$. We denote by $A_0$ the matrix of $[-,-]$ with respect to the basis~\eqref{eq-Wbasis}.

For a $\BZ$-algebra $R$, let $W_R \defeq W \otimes_\BZ R$. For a field $k$ of characteristic not equal to $2$, the $k$-vector space $W_k$ equipped with the bilinear form $[-,-]$ is called the \emph{split orthogonal space} of dimension $2n+1$ and determinant $(-1)^n$ over $k$ and is unique up to $k$-isomorphism~\cite[\S6]{MR0506372}. The space $W_k$ is called ``split'' because it is a nondegenerate quadratic space containing a maximal isotropic subspace (defined over $k$) for the bilinear form $[-,-]$.

Let $T \in \End_R(W_R)$. Recall that the adjoint transformation $T^\dagger \in \End_R(W_R)$ of $T$ with respect to the form $[-,-]$ is determined by the formula $[Tv, w] = [v, T^\dagger w]$ for all $v,w \in W_R$, and that $T$ is said to be \emph{self-adjoint} if $T = T^\dagger$. If $T$ is expressed as a matrix with respect to the basis~\eqref{eq-Wbasis}, then $T$ is self-adjoint with respect to the form $[-,-]$ if and only if $T^TA_0 = A_0T$.

Let $V$ be the affine space defined over $\Spec \BZ$ whose $R$-points are given by
$$V(R) = \{T \in \End_R(W_R) : T = T^\dagger \}$$
for any $\BZ$-algebra $R$.
An $R$-automorphism $g \in \Aut_R(W_R)$ is called \emph{orthogonal} with respect to the form $[-,-]$ if $[g(v),g(w)] = [v,w]$ for all $v,w \in W_R$. If $g$ is expressed as a matrix with respect to the basis~\eqref{eq-Wbasis}, then $g$ is orthogonal with respect to the form $[-,-]$ if and only if $g^TA_0g = A_0$. Let $G \defeq \on{SO}(W)$ be the group scheme defined over $\Spec \BZ$ whose $R$-points are given by
$$G(R) = \{g \in \Aut_R(W_R) : g\text{ is orthogonal with respect to $[-,-]$ and }\det g = 1\}$$
for any $\BZ$-algebra $R$. The group scheme $G$ is known as the \emph{split odd special orthogonal group} of the lattice $W$, and it acts on $V$ by conjugation: for $g \in G(R)$ and $T \in V(R)$, the map $gTg^{-1} \in \End_R(W_R)$ is readily checked to be self-adjoint with respect to $[-,-]$ and is therefore an element of $V(R)$. We thus obtain a linear representation $G \to \on{GL}(V)$ of dimension $\dim V = 2n^2 + 3n + 1$.

Since $G$ acts on $V$ by conjugation, the characteristic polynomial
\begin{equation*} 
\on{ch}(T) \defeq \det(x \cdot \id - T) = x^{2n+1} + \sum_{i = 1}^{2n+1} c_i x^{2n+1-i} \in R[x]
\end{equation*}
is invariant under the action of $G(R)$ for each $T \in V(R)$. In fact, by~\cite[\S8.3, part (VI) of \S13.2]{MR0453824}, the coefficients $c_1, \dots, c_{2n+1}$, which have respective degrees $1, \dots, 2n+1$, freely generate the ring of $G$-invariant functions on $V$.


\subsection{Rings Associated to Binary Forms} \label{sec-ringsbins}
For the rest of \S\ref{sec-algprems}, let $R$ be a principal ideal domain (we typically take $R$ to be $\BZ$, $\BQ$, $\BZ_p$, $\BZ/p\BZ$, or $\BQ_p$, where $p \in \BZ$ is a prime), and let $k$ be the fraction field of $R$. Let $n \geq 1$, and let $$F(x,z) = \sum_{i = 0}^{2n+1} f_i x^{2n+1-i}z^i \in R[x,z]$$ be a separable binary form of degree $2n+1$ with leading coefficient $f_0 \in R\smallsetminus\{0\}$. 

Consider the \'{e}tale $k$-algebra $$K_F \defeq k[x]/(F(x,1)).$$
Let $\theta$ denote the image of $x$ in $K_F$. For each $i \in \{1, \dots, 2n\}$, let $p_i$ be the polynomial defined by
$$p_i(t) \defeq \sum_{j = 0}^{i-1} f_j t^{i-j},$$
and let $\zeta_i \defeq p_i(\theta)$. To the binary form $F$, there is a naturally associated free $R$-submodule $R_F \subset K_F$ having rank $2n+1$ and $R$-basis given by
\begin{equation} \label{eq-rfbasis}
R_F \defeq R \langle 1, \zeta_1, \zeta_2, \dots, \zeta_{2n} \rangle.
\end{equation}
The module $R_F$ is of significant interest in its own right and has been studied extensively in the literature. In~\cite[proof of Lemma~3]{MR0306119}, Birch and Merriman proved that the discriminant of $F$ is equal to the discriminant of $R_F$, and in~\cite[Proposition 1.1]{MR1001839}, Nakagawa proved that $R_F$ is actually a ring (and hence an order in $K_F$) having multiplication table
\begin{equation} \label{eq-multtable}
\zeta_i\zeta_j = \sum_{k = j+1}^{\min\{i+j,2n+1\}} f_{i+j-k}\zeta_k - \sum_{k = \max\{i+j-(2n+1),1\}}^i f_{i+j-k}\zeta_k,
\end{equation}
where $1 \leq i \leq j \leq 2n$ and we take $\zeta_0 = 1$ and $\zeta_{2n+1} = -f_{2n+1}$ for the sake of convenience. (To be clear, these results of Nakagawa are stated for the case of irreducible $F$, but as noted in~\cite[\S2.1]{MR2763952}, their proofs continue to hold when $F$ is not irreducible.)

Also contained in $K_F$ is a natural family of free $R$-submodules $I_F^j$ of rank $2n+1$ for each $j \in \{0, \dots, 2n\}$, having $R$-basis given by
\begin{equation} \label{eq-idealdef}
I_F^j \defeq R\langle1,\theta, \dots, \theta^j,\zeta_{j+1}, \dots, \zeta_{2n} \rangle.
\end{equation}
Note that $I_F^0 = R_F$ is the unit ideal. By~\cite[Proposition A.1]{MR2763952}, each $I_F^j$ is an $R_F$-module and hence a fractional ideal of $R_F$; moreover, the notation $I_F^j$ makes sense, because $I_F^j$ is equal to the $j^{\mathrm{th}}$ power of $I_F^1$. By~\cite[Proposition A.4]{MR2763952}, the fractional ideals $I_F^j$ are invertible precisely when the form $F$ is primitive, in the sense that $\sum_{i = 0}^{2n+1}Rf_i = R$. It is a result of Simon (see~\cite[Proposition~14]{MR2523319}, cf.~\cite[Theorem~2.4]{MR2763952}) that $$J_F \defeq I_F^{2n-1}$$ represents the ideal class of the inverse different of $R_F$.

Given a fractional ideal $I$ of $R_F$ having a specified basis (i.e., a \emph{based fractional ideal}), the \emph{norm} of $I$, denoted by $\on{N}(I)$, is defined to be the determinant of the $k$-linear transformation taking the basis of $I$ to the basis of $R_F$ in~\eqref{eq-rfbasis}. It is easy to check that $\on{N}(I_F^j) = f_0^{-j}$ for each $j$ with respect to the basis in~\eqref{eq-idealdef}. The norm of $\kappa \in K_F^\times$ is $\on{N}(\kappa) \defeq \on{N}(\kappa \cdot I_F^0)$ with respect to the basis $\langle \kappa, \kappa \cdot \zeta_1, \dots, \kappa \cdot \zeta_{2n}\rangle$ of $\kappa \cdot I_F^0$. Note that we have the multiplicativity relation
\begin{equation} \label{eq-normsmult}
\on{N}(\kappa \cdot I) = \on{N}(\kappa) \cdot \on{N}(I)
\end{equation}
for any $\kappa \in K_F^\times$ and fractional ideal $I$ of $R_F$ with a specified basis.

We now explain how the objects $K_F$, $R_F$, and $I_F^j$ transform under the action of $\gamma \in \on{SL}_2(R)$ on binary forms of degree $2n+1$ defined by $\gamma \cdot F = F((x,z) \cdot \gamma)$. If $F' = \gamma \cdot F$, then $\gamma$ induces an isomorphism $K_F \simeq K_{F'}$, under which the rings $R_F$ and $R_{F'}$ happen to be identified with each other (see~\cite[Proposition 1.2]{MR1001839} for a direct proof and~\cite[\S2.3]{MR2763952} for a geometric argument). On the other hand, the ideals $I_F^j$ and $I_{F'}^j$ are isomorphic as $R_F$-modules but may \emph{not necessarily} be identified under the isomorphism $K_F \simeq K_F'$. Indeed, as explained in~\cite[(7)]{thesource}, these ideals are related by the following explicit rule: if $\gamma = \left[\begin{smallmatrix} a & b \\ c & d\end{smallmatrix}\right]$, then for each $j \in \{0, \dots, 2n\}$, the composition
\begin{equation} \label{eq-idealident}
\begin{tikzcd}
I_F^j \arrow[hook]{r}{\phi_{j,\gamma}} &  K_F \arrow{r}{\sim} & K_{F'}
\end{tikzcd}
\end{equation}
is an injective map of $R_F$-modules having image equal to $I_{F'}^{j}$, where $\phi_{j,\gamma}$ sends each $\delta \in I_F^j$ to $(-b\theta+a)^{-j} \cdot \delta \in K_F$. When $j = 0$, we recover the identification of $R_F = I_F^0$ \mbox{with $R_{F'} = I_{F'}^0$ from~\eqref{eq-idealident}.}

\subsection{Parametrization of Square Roots of the Class of $J_F$} \label{sec-param}

We say that a based fractional ideal $I$ of $R_F$ is a \emph{square root of the class of the inverse different} if there exists $\delta \in K_F^\times$ such that
\begin{equation} \label{eq-sqrtconds}
    I^2 \subset \delta \cdot J_F\quad \text{and}\quad  \on{N}(I)^2 = \on{N}(\delta) \cdot \on{N}(J_F)
\end{equation}
We now recall an orbit parametrization for such square roots that we introduced in~\cite{Swpreprint}. To do this, let $\wt{H}_F$ be the set of pairs $(I, \delta)$ satisfying~\eqref{eq-sqrtconds}. Let $H_F = \wt{H}_F/\hspace*{-2pt}\sim$, where the equivalence relation $\sim$ is defined as follows: $(I_1, \delta_1) \sim (I_2, \delta_2)$ if and only if there exists $\kappa \in K_F^\times$ such that the based fractional ideals $I_1$ and $\kappa \cdot I_2$ are equal up to an $\on{SL}_{2n+1}(R)$-change-of-basis and such that $\alpha_1 = \kappa^2 \cdot \alpha_2$. Now, take $(I, \delta) \in H_F$, and consider the symmetric bilinear form
\begin{equation} \label{eq-defbilin}
\langle-,-\rangle \colon I \times I \to K_F, \quad (\alpha, \beta) \mapsto \langle \alpha, \beta \rangle =\delta^{-1} \cdot \alpha\beta.
\end{equation}
Let $\pi_{2n-1},\pi_{2n} \in \Hom_{R}(J_F,R)$ be the maps defined on the $R$-basis~\eqref{eq-idealdef} of $J_F$ by
\begin{align*}
& \pi_{2n-1}(\theta^{2n-1}) - 1 = \pi_{2n-1}(\zeta_{2n}) = \pi_{2n-1}(\theta^i) = 0 \text{ for each } i \in \{0, \dots, 2n-2\} , \text{ and}\\
& \pi_{2n}(\zeta_{2n}) + 1 = \pi_{2n}(\theta^i) = 0 \text{ for each } i \in \{0, \dots, 2n-1\}.
\end{align*}
Let $A$ and $B$ respectively denote the matrices representing the symmetric bilinear forms $\pi_{2n} \circ \langle -, - \rangle \colon I \times I \to R$ and $\pi_{2n-1} \circ \langle -, - \rangle \colon I \times I \to R$ with respect to the $R$-basis of $I$, and observe that $A$ and $B$ are symmetric of dimension $(2n+1) \times (2n+1)$ and have entries in $R$. We have thus constructed a map of sets
\begin{equation*}
\mathsf{orb}_F \colon H_F \longrightarrow \left\{\text{$\on{SL}_{2n+1}(R)$-orbits on $R^2 \otimes_R \Sym_2 R^{2n+1}$}\right\},
\end{equation*}
where by $R^2 \otimes_R \on{Sym}_2 R^{2n+1}$ we mean the space of pairs of $(2n+1) \times (2n+1)$ symmetric matrices with entries in $R$. The following result characterizes the map $\mathsf{orb}_F$:
\begin{theorem}[\protect{\cite[Theorems~14 and 18, Proposition~17 and 20]{Swpreprint}}] \label{thm-theconstruction}
The map $\mathsf{orb}_F$ defines a bijection from $H_F$ to the set of all those $\on{SL}_{2n+1}(R)$-orbits of pairs $(A,B) \in R^2 \otimes_R \Sym_2 R^{2n+1}$ such that:
\begin{enumerate}
\item[$(\mathrm{a})$] $(-1)^n \cdot \det(x \cdot A + z \cdot B) = F_{\on{mon}}(x,z) \defeq \frac{1}{f_0}\cdot F(x,f_0z)$; and
\item[$(\mathrm{b})$] $p_i\big(\frac{1}{f_0} \cdot -A^{-1}B\big)$ is a matrix with entries in $R$ for each $i \in \{1, \dots, 2n\}$.
\end{enumerate}
Moreover, for any $(I, \delta) \in H_F$, the stabilizer in $\on{SL}_{2n+1}(R)$ of $\mathsf{orb}_F(I, \delta)$ contains a subgroup isomorphic to the group $R_F^\times[2]_{\on{N}\equiv 1} \defeq \{\rho \in R_F^\times : \rho^2 = 1 = \on{N}(\rho)\}$.
\end{theorem}
For the purposes of this paper, it turns out (see Lemma~\ref{lem-split}) that we only need the restriction of the parametrization in Theorem~\ref{thm-theconstruction} to those pairs $(A,B)$ such that $A$ defines a split quadratic form over $k$. This restricted parametrization can be expressed in terms of the action of $G(R)$ on $V(R)$. Indeed, suppose we are given a pair $(A,B) \in R^2 \otimes_{R} \Sym_2 R^{2n+1}$ such that the \emph{invariant form} $\det(x \cdot A + z \cdot B)$ of $(A,B)$ is equal to $(-1)^n \cdot F_{\on{mon}}(x,z)$ and such that $A$~defines a quadratic form that is split over $k$. Then there exists $g \in \on{GL}_{2n+1}(R)$ satisfying $gAg^T = A_0$ (by the uniqueness statement in the first paragraph of \S\ref{sec-rep}), and the matrix $T = (g^T)^{-1}\cdot -A^{-1}B \cdot g^T$ is self-adjoint with respect to the matrix $A_0$ and has characteristic polynomial $F_{\on{mon}}(x,1)$. We thus obtain the following correspondence, the proof of which is immediate:

\begin{proposition} \label{prop-slo}
The map $(A,B) \mapsto T$ described above defines a natural bijection from the set of $\on{SL}_{2n+1}(R)$-orbits on pairs $(A,B) \in R^2 \otimes_R \on{Sym}_2 R^{2n+1}$ with invariant form $(-1)^n \cdot F_{\on{mon}}$ such that $A$ defines a split quadratic form over $k$ to the set of $G(R)$-orbits on $V(R)$ with characteristic polynomial $F_{\on{mon}}(x,1)$. Under this bijection, we have that $\on{Stab}_{\on{SL}_{2n+1}(R)}(A,B) = \on{Stab}_{G(R)}(T)$.
\end{proposition}
\begin{remark}
As explained in~\cite[\S2.2]{Swpreprint}, the proof of Theorem~\ref{thm-theconstruction} is inspired by~\cite{MR3187931}, where Wood derives a similar parametrization in which the multiplication table of a fractional ideal of $R_F$ gives rise to a pair of symmetric matrices with invariant form equal to $F$ (up to a sign). In~\cite[\S2]{thesource}, Bhargava uses Wood's parametrization to prove analogues of Theorems~\ref{thm-main} and~\ref{thm-hasse} for rational points on hyperelliptic curves.
\end{remark}

\subsection{Orbits over Fields} \label{sec-fieldgens}

When $R = k$ is a field, the parametrization in Theorem~\ref{thm-theconstruction} simplifies considerably. Indeed, every fractional ideal of $R_F = K_F$ is equal to $K_F$, so every element of $H_F$ is of the form $(K_F,\delta)$, where $f_0 \cdot \delta \in (K_F^\times/K_F^{\times2})_{\on{N}\equiv1} \defeq \{ \rho \in K_F^\times : \on{N}(\rho) \in k^{\times2}\}/K_F^{\times2}$. Moreover, condition $\mathrm{(b)}$ in Theorem~\ref{thm-theconstruction} holds trivially in this case, implying that every $\on{SL}_{2n+1}(k)$-orbit on $k^2 \otimes_k \on{Sym}_2 k^{2n+1}$ with invariant form $F_{\on{mon}}(x,1)$ arises from an element of $H_F$. We thus obtain the following result, which gives a convenient description of the orbits of $\on{SL}_{2n+1}(k)$ on $k^2 \otimes_k \on{Sym}_2 k^{2n+1}$ that arise via the parametrization:

\begin{proposition}[\protect{\cite[Proposition~36]{Swpreprint}}] \label{prop-rationalcorresp}
When $R = k$ is a field, the set of $\on{SL}_{2n+1}(k)$-orbits on $k^2 \otimes_k \on{Sym}_2 k^{2n+1}$ with invariant form $F_{\on{mon}}$ is in bijection with the set $(K_F^\times/K_F^{\times 2})_{\on{N} \equiv 1}$. Under this bijection, the orbit corresponding to the pair $(K_F, \delta) \in H_F$ is sent to $f_0 \cdot \delta$.
\end{proposition}
In what follows, we say that an orbit of $\on{SL}_{2n+1}(k)$ on $k^2 \otimes_k \on{Sym}_2 k^{2n+1}$ (or of $G(k)$ on $V(k)$) is \emph{distinguished} if it corresponds to $1 \in (K_F^\times/K_F^{\times2})_{\on{N}\equiv1}$ under the bijection of Proposition~\ref{prop-rationalcorresp}. In geometric terms, the orbit of a pair $(A,B) \in k^2 \otimes_k \on{Sym}_2 k^{2n+1}$ is distinguished if and only if $A$ and $B$ share a maximal isotropic space defined over $k$ (see~\cite[\S4]{MR3156850}).

 For ease of notation, write $f = F_{\on{mon}}$. Given any $(A,B) \in k^2 \otimes_k \on{Sym}_2 k^{2n+1}$ with invariant form $f$, we may regard $A$ and $B$ as quadratic forms cutting out a pair of quadric hypersurfaces $Q_A$ and $Q_B$ in $\BP_k^{2n}$. Let $\scr{F}_{(A,B)}$ denote the Fano scheme parametrizing $(n-1)$-dimensional linear spaces contained in the base locus $Q_A \cap Q_B$ of the pencil of quadric hypersurfaces spanned by $Q_A$ and $Q_B$. The following proposition states that $\scr{F}_{(A,B)}$ is a torsor of the Jacobian of $\mathscr{S}_F$:


\begin{proposition} \label{prop-torsor}
Let $(A,B) \in k^2 \otimes_k \on{Sym}_2 k^{2n+1}$ with invariant form $f$. Then, if $\#k$ is sufficiently large, $\scr{F}_{(A,B)}$ is a torsor over $k$ of the Jacobian of the stacky curve $\mathscr{S}_F$.
\end{proposition}
\begin{proof}
The Jacobian of a stacky curve $\scr{X}$ is defined to be the $0^{\mathrm{th}}$-degree component $\on{Pic}^0(\scr{X})$ of the group $\on{Pic}(\scr{X})$ of divisors on $\scr{X}$ modulo principal divisors. Let $k_{\on{sep}}$ be a separable closure of $k$, and let $\theta_1, \dots, \theta_{2n+1} \in \BP^1_k(k_{\on{sep}})$ be the roots of $F$. Applying the correspondence between isomorphism classes of line bundles and divisor classes to~\cite[proof of Proposition 2.2]{stax} yields that $\on{Pic}^0(\mathscr{S}_F)$ is the group generated by the classes of the divisors $(\tfrac{1}{2}\theta_i - \tfrac{1}{2}\theta_j)$ subject to the relations $2 \cdot (\tfrac{1}{2}\theta_i - \tfrac{1}{2}\theta_j) = 0$ for $1 \leq i < j \leq 2n+1$; i.e., we have
$$\on{Pic}^0(\mathscr{S}_F) =  (\BZ/2\BZ) \langle (\tfrac{1}{2}\theta_i - \tfrac{1}{2}\theta_j) : 1 \leq i < j \leq 2n+1 \rangle.$$
There is a natural right action of the absolute Galois group $G_k \defeq \on{Gal}(k_{\on{sep}}/k)$ on $\on{Pic}^0(\mathscr{S}_F)$: for $\sigma \in G_k$, we have $(\tfrac{1}{2} \theta_i - \tfrac{1}{2} \theta_i) \cdot \sigma = (\tfrac{1}{2} (\sigma^{-1} \cdot \theta_i) - \tfrac{1}{2} (\sigma^{-1} \cdot \theta_j))$.  

A result of Wang (see~\cite[Corollary 2.5]{MR3800357}) implies that, as long as $\#k$ is sufficiently large, the Fano scheme $\scr{F}_{(A,B)}$ is a torsor over $k$ of the $2$-torsion subgroup $J[2]$ of the Jacobian $J$ of the monic odd-degree hyperelliptic curve $y^2 = f(x,1)$. Thus, to prove that $\scr{F}_{(A,B)}$ is a torsor of $\on{Pic}^0(\mathscr{S}_F)$, it suffices to show that there is a $G_k$-equivariant isomorphism $\on{Pic}^0(\mathscr{S}_F) \simeq J[2]$. But this is clear: the divisor classes $(\theta_i - \theta_j)$ for $1 \leq i < j \leq 2n+1$ generate $J[2](k_{\on{sep}})$ over $\BZ/2\BZ$.
\end{proof}
Note in particular that the $\on{SL}_{2n+1}(k)$-orbit of $(A,B)$ is distinguished if and only if $\scr{F}_{(A,B)}$ is the trivial torsor of the Jacobian of $\mathscr{S}_F$ (which happens if and only if $\scr{F}_{(A,B)}(k) \neq \varnothing$).
\begin{remark}
 As it happens, Proposition~\ref{prop-torsor} admits an analogue for pencils of even-dimensional quadric hypersurfaces; see the work of Wang~\cite{MR3800357} (generalizing previous work of Reid~\cite[Theorem 4.8]{Reid72}), showing that the Fano scheme of maximal linear spaces in the base locus of a pencil of quadric hypersurfaces generated by a pair $(A,B) \in k^2 \otimes_k \on{Sym}_2 k^{2n}$ is a torsor of the Jacobian of the hyperelliptic curve cut out by the equation $y^2 = (-1)^n \cdot \det(x \cdot A+ z \cdot B)$.
\end{remark}


We conclude this section by explaining how to visualize the simply transitive action of $\on{Pic}^0(\mathscr{S}_F)$ on $\scr{F}_{(A,B)}(k_{\on{sep}})$ from Proposition~\ref{prop-torsor}. For each $i \in \{1, \dots, 2n+1\}$, let $Q_i$ be the singular fiber lying over the point $\theta_i$ in the pencil of quadrics spanned by $Q_A$ and $Q_B$. Suppose that $T$ is generic, so that each of the quadrics $Q_i$ is a simple cone with cone point $q_i$. Let $\scr{F}_{(A,B)}(k_{\on{sep}}) = \{p_1, \dots, p_{2^{2n}}\}$, and for each $\ell \in \{1, \dots, 2^{2n}\}$, let $\ol{p}_\ell$ denote the corresponding linear subspace of $\BP_{k_{\on{sep}}}^{2n}$. We illustrate this setup in the case where $n = 1$ in Figure~\ref{fig-colorista}.

We define an action of $\on{Pic}^0(\mathscr{S}_F)$ on $\scr{F}_{(A,B)}(k_{\on{sep}})$ by first specifying how each $(\tfrac{1}{2}\theta_i)$ acts on $\scr{F}_{(A,B)}(k_{\on{sep}})$. Let $L_{i\ell} \subset \BP_{k_{\on{sep}}}^{2n}$ be the linear span of the $(n-1)$-plane $\ol{p}_\ell$ and the point $q_i$ for each pair $(i,\ell)$. By~\cite[Theorem~3.8]{Reid72}, there exists precisely one element $p_{i(\ell)} \in \scr{F}_{(A,B)}(k_{\on{sep}}) \smallsetminus \{p_\ell\}$ such that $\ol{p}_{i(\ell)} \subset L_{i\ell}$. For each $i$, let $(\tfrac{1}{2}\theta_i)$ act on $\scr{F}_{(A,B)}$ by swapping $p_\ell$ and $p_{i(\ell)}$ for each $\ell$. Then for each pair $(i,j)$, the divisor $(\tfrac{1}{2}\theta_i - \tfrac{1}{2}\theta_j)$ acts on $\scr{F}_{(A,B)}$ by swapping $p_\ell$ with $p_{i(j(\ell))} = p_{j(i(\ell))}$ for each $\ell$. When $n = 1$, for example, the divisor $(\tfrac{1}{2}\theta_1 - \tfrac{1}{2}\theta_2)$ acts on $\scr{F}_{(A,B)}$ by sending $(p_1, p_2, p_3, p_4)$ to $(p_3, p_4, p_1, p_2)$, as we demonstrate in Figure~\ref{fig-colorista}.
\FloatBarrier
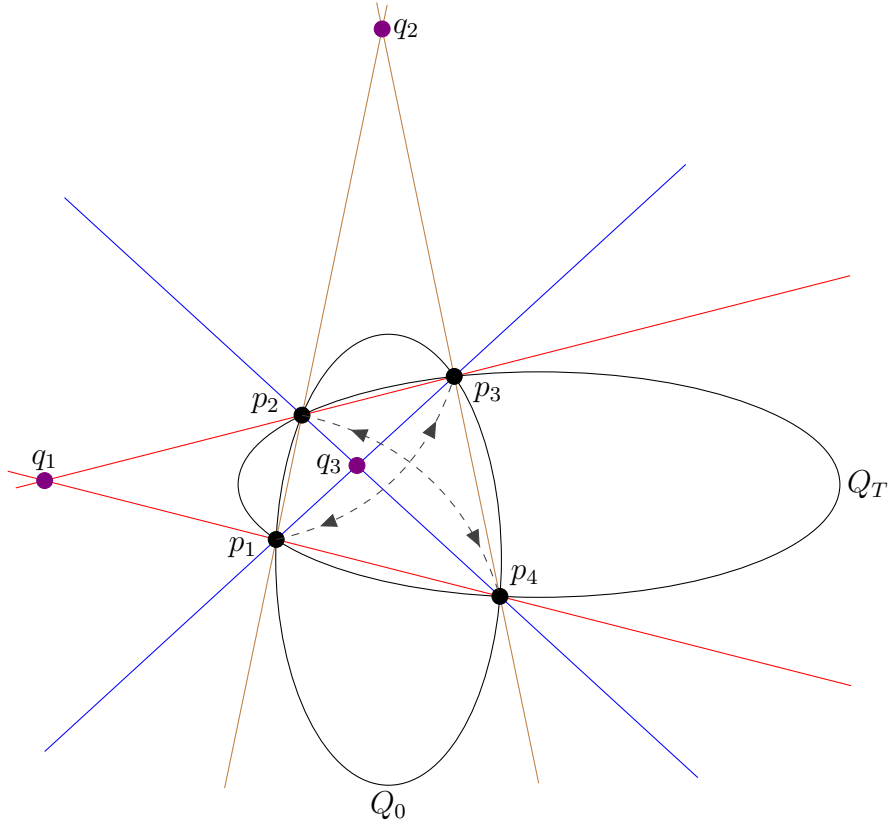
\begin{figure}
\centering
   \begin{tikzpicture}
   \draw (0,0) ellipse (4cm and 1.5cm);
   \draw (-2,-4.6) node[anchor=south]{$Q_A$};
   \draw (-2,-1) ellipse (1.5cm and 3 cm);
   \draw (4.8,0) node[anchor=east]{$Q_B$};
   \draw [add= 1.2 and 1.57, red] (-3.49393,-0.730286) to (-0.519919,-1.48727);
   \draw [add= 1.88 and 2.6, red] (-3.15087,0.924055) to (-1.12678,1.43926);
   \draw (-2.46989,0.274096) node[anchor=east]{$q_3$};
   \draw [add= 1.3 and 1.3, blue] (-3.49393,-0.730286) to (-1.12678,1.43926);
   \draw [add= 1.2 and 1, blue] (-3.15087,0.924055) to (-0.519919,-1.48727);
   \draw (-6.57213,0.0532183) node[anchor=south]{$q_1$};
   \draw [add= 2 and 3.3, brown] (-3.49393,-0.730286) to (-3.15087,0.924055);
   \draw [add= 1.7 and 0.85, brown] (-1.12678,1.43926) to (-0.519919,-1.48727);
   \draw (-2.0854,6.0621) node[anchor=west]{$q_2$};
   \filldraw (-3.49393,-0.730286) circle (3pt);
   \filldraw (-3.15087,0.924055) circle (3pt);
   \filldraw (-1.12678,1.43926) circle (3pt);
   \filldraw (-0.519919,-1.48727) circle (3pt);
      \draw (-3.60393,-0.560286) node[anchor=north east]{$p_1$};
   \draw (-3.30087,0.804055) node[anchor= south east]{$p_2$};
   \draw (-1.00678,1.49926) node[anchor = north west]{$p_3$};
   \draw (-0.519919,-1.48727) node[anchor = south west]{$p_4$};
   \filldraw [violet] (-2.0854,6.0621) circle (3pt);
   \filldraw [violet] (-6.57213,0.0532183) circle (3pt);
   \filldraw [violet] (-2.41989,0.254096) circle (3pt);
   \path [draw=darkgray,dashed,postaction={on each segment={rightend arrow=darkgray, leftend arrow = darkgray}}]
  (-3.49393,-0.730286) to [bend right] (-1.12678,1.43926)
  (-3.15087,0.924055) to [bend left] (-0.519919,-1.48727);
   \end{tikzpicture}
   \caption{The pencil of conics spanned by $Q_A$ and $Q_B$ in $\BP_{k_{\on{sep}}}^2$, with singular fibers $Q_1$, $Q_2$, and $Q_3$ having cone points $q_1$, $q_2$, and $q_3$, respectively. In this case, $\scr{F}_{(A,B)}(k_{\on{sep}}) = Q_A \cap Q_B = \{p_1, p_2, p_3, p_4\}$. The dashed arrows display the action of the degree-$0$ divisor $(\tfrac{1}{2}\theta_1 - \tfrac{1}{2}\theta_2)$ on $\scr{F}_{(A,B)}$.}
      \label{fig-colorista}
\end{figure}

\section{Construction of an Integral Orbit from a Primitive Integer Solution} \label{sec-buildabear}

Take $R$ and $F$ as in \S\ref{sec-ringsbins}, and assume that $k$ is of characteristic $0$. In \S\ref{sec-theconstrsec}, we prove that points $(x_0,y_0,z_0) \in \mathscr{S}_F(R)$ give rise to certain orbits of $G(R)$ on $V(R)$ with characteristic polynomial $F_{\mathrm{mon}}(x,1)$. To produce the desired orbit, it suffices by Theorem~\ref{thm-theconstruction} and Proposition~\ref{prop-slo} to construct a pair $(I, \delta) \in H_F$ such that the matrix $A$ that arises defines a split quadratic form over $k$. Our construction extends a result of Simon, who constructs the desired pair $(I,\delta)$ under the following special conditions: $F$ is primitive and irreducible and $y_0$ is coprime to the index of $R_F$ in its integral closure (see~\cite[Corollary~4]{MR2523319}). We finish in \S\ref{sec-whosdist} by describing when the orbits arising from the construction are distinguished.

\subsection{The Construction} \label{sec-theconstrsec} Let $(x_0,y_0,z_0) \in \mathscr{S}_F(R)$. For convenience, we say that the point $(x_0, y_0, z_0)$ is a \emph{Weierstrass point} if $y_0 = 0$ and a \emph{non-Weierstrass point} otherwise. Note that if $y_0 = 0$, then $z_0 \neq 0$ because $f_0 \neq 0$, so the binary form $F$ factors uniquely as
\begin{equation} \label{eq-defftilde}
F(x,z) = \big(x - \tfrac{x_0}{z_0} \cdot z \big) \cdot \wt{F}(x,z)
\end{equation}
where $\wt{F} \in R[x,z]$ is a binary form of degree $2n$ with leading coefficient equal to $f_0$.

We start by transforming the pair $(x_0, z_0)$ into a pair with simpler coordinates by constructing a suitable $\gamma \in \on{SL}_2(R)$ such that $(x_0, z_0) \cdot \gamma = (0,1)$. Let $b_0, d_0 \in R$ be any elements such that $b_0x_0 + d_0z_0 = 1$ (such $b_0, d_0$ exist since $Rx_0 + Rz_0 = R$). We claim that there exists $t \in R$ such that $(d_0 -t x_0) + \theta(b_0 + t z_0) \in K_F^\times$. Indeed, if we were to have $(d_0 -t x_0) + \theta(b_0 + t z_0) \in K_F \smallsetminus K_F^\times$, then the pair $(x,z) = (d_0 - tx_0,-b_0 - t z_0)$ is a root of the equation $F(x,z) = 0$, which has finitely many roots up to scaling, implying that we can choose $t$ to avoid this outcome (since $k$ is of characteristic $0$). Take any such $t \in R$, let $b = b_0 + t z_0$ and $d = d_0 - t x_0$, \mbox{and let}
$\gamma = \left[\begin{smallmatrix} z_0 & b\\ -x_0 & d\end{smallmatrix}\right]$. Now, let $F'$ be defined by
 $$F'(x,z) = F ((x,z) \cdot \gamma^{-1}) = \sum_{i = 0}^{2n+1} f_i'  x^{2n+1-i}z^i \in R[x,z].$$
Note that $f_{2n+1}' = y_0^2$ and that $f_0' \neq 0$ since $f_0' = f_0 \cdot \on{N}(d + \theta b)$. If $y_0 = 0$, the binary form $F'$ \mbox{factors uniquely as}
$$F'(x,z) = x \cdot \wt{F}'(x,z),$$
where $\wt{F}' \in R[x,z]$ is a binary form of degree $2n$ with leading coefficient equal to $f_0'$; further observe that $f_{2n}' = \wt{F}(x_0,z_0) \neq 0$ because $F$ is separable. Let $\theta' = (-x_0 + \theta z_0)/(d+\theta b)$ be the root of $F'(x,1)$ in $K_{F'} = K_F$, and consider the free rank-$(2n+1)$ $R$-submodule $I' \subset K_{F'}$ defined by
\begin{equation} \label{eq-I'isdef}
I' \defeq R\langle \xi, \theta', \dots, {\theta'}^n, \zeta_{n+1}', \dots, \zeta_{2n}'\rangle,  \quad \text{where} \quad \xi \defeq\begin{cases} y_0 & \text{ if $y_0 \neq 0$,} \\ \zeta_{2n}'+f_{2n}' & \text{ if $y_0 = 0$.} \end{cases}
\end{equation}
Observe that the $R$-module $I'$ is a fractional ideal of $R_{F'}$ that resembles the fractional ideal $I_{F'}^n$, the only difference between them being that the basis element $1 \in I_{F'}^n$ is replaced by $\xi \in I'$.
\begin{lemma} \label{lem-idprime}
We have that
$${I'}^2 \subset \delta' \cdot J_{F'}, \quad \text{where} \quad \delta' \defeq \begin{cases} -\theta' & \text{ if $y_0 \neq 0$,}\\ \wt{F}'(\theta',1)-\theta'  & \text{ if $y_0 = 0$}\end{cases}$$
and that $\on{N}(I')^2 = \on{N}(\delta') \cdot \on{N}(J_{F'})$.
\end{lemma}
\begin{proof}
Recall from~\eqref{eq-idealdef} that the fractional ideal $J_{F'}$ has the $\BZ$-basis
$$J_{F'} = R\langle1 , \theta', \dots, \theta'^{2n-1}, \zeta_{2n}' \rangle.$$
To prove the claimed containment, it suffices to check that when each of the pairwise products of the basis elements of $I'$ in~\eqref{eq-I'isdef} is divided by $\delta'$, what results is an element of $J_{F'}$. We have the following computations:
\begin{align*}
 \xi^2/\delta' & =  \zeta_{2n}' + f_{2n}' & & \\
  (\xi \cdot {\theta'}^i)/\delta' & = -y_0 {\theta'}^{i-1} & & \text{for $i \in \{1, \dots, n\}$} \\
 (\xi \cdot \zeta_j')/\delta' & =  -y_0  \zeta_{j-1}' - y_0 f_{j-1}' &  & \text{for }j \in \{n+1, \dots, 2n\} \\
 ({\theta'}^i \cdot {\theta'}^j)/\delta' & = -{\theta'}^{i+j-1} &  & \text{for }i,j \in \{1, \dots, n\} \\
 ({\theta'}^i \cdot \zeta_j')/\delta' & = -{\theta}'^{i-1} \zeta_j' &  & \text{for }i \in \{1, \dots, n\},\,j \in \{n+1, \dots, 2n\} \\
 (\zeta_i'\cdot \zeta_j')/\delta' & = -\zeta_i'\zeta_{j-1}'-f_{j-1}'\zeta_i' &  & \text{for }i,j \in \{n+1, \dots, 2n\}
 \end{align*}
Using the fact that $\zeta_i' \in R_{F'}$ for each $i \in \{0, \dots, 2n+1\}$ together with the fact that $J_{F'}$ is closed under multiplication by elements in $R_{F'}$, one readily verifies that each of the expressions obtained on the right-hand side above is an element of $J_{F'}$. Crucially, these expressions do not depend in piecewise fashion on the value of $y_0$, even though $\xi$ and $\delta'$ do (see Remark~\ref{rmk-thefix}).

As for the second part of the lemma, using $\on{N}(I_{F'}^n) = {f_0'}^{-n}$, we find that $\on{N}(I') = \wt{\xi} \cdot {f_0'}^{-n}$, where $\wt{\xi} = \xi = y_0$ if $y_0 \neq 0$ and $\wt{\xi} = \xi-\zeta_{2n}' = f_{2n}'$ if $y_0 = 0$. Moreover, a calculation reveals that $\on{N}(\delta') = \wt{\xi}^2 \cdot {f_0'}^{-1}$. Combining these norm calculations, we deduce that
\begin{equation*}
\on{N}(I')^2 =  (\wt{\xi} \cdot {f_0'}^{-n})^2 = (\wt{\xi}^2 \cdot {f_0'}^{-1}) \cdot ({f_0'}^{-(2n-1)}) = \on{N}(\delta') \cdot \on{N}(J_{F'}) \qedhere
\end{equation*}
\end{proof}
We now transform $I'$ and $\delta'$ back from $R_{F'}$ to $R_F$. Recall from~\eqref{eq-idealident} that $$I_F^{n} = \phi_{n,\gamma}(I_{F'}^n) = (-b \theta' + z_0)^{-n} \cdot I_{F'}^n,$$
where we can invert $-b\theta' + z_0$ because its inverse is equal to $d + \theta b$. Since $I'$ resembles $I_{F'}^n$, it makes sense to consider the fractional ideal
\begin{equation} \label{eq-defofIinverts}
I \defeq \phi_{n,\gamma}(I') = (-b \theta' + z_0)^{-n} \cdot I' = (d + \theta b)^n \cdot I',
\end{equation}
which has norm given by
\begin{align} \label{eq-normofI}
    \on{N}(I) & = \on{N}((d + \theta b)^n \cdot I') = \on{N}(d + \theta b)^n \cdot \on{N}(I') = \wt{\xi}\cdot f_0^{-n}
\end{align}
where we used the multiplicativity relation~\eqref{eq-normsmult}. Similarly, let $\delta \in K_F^\times$ be defined as follows:
\begin{equation} \label{eq-inspect}
\delta \defeq \phi_{1,\gamma}(\delta') = \frac{\delta'}{-b\theta'+z_0} = \begin{cases} x_0 - \theta z_0 & \text{ if $y_0 \neq 0$,} \\ \wt{F}(z_0\theta, z_0) + (x_0 - \theta z_0) & \text{ if $y_0 = 0$.} \end{cases}
\end{equation}
It then follows from Lemma~\ref{lem-idprime} that
\begin{equation} \label{eq-id}
I^2 \subset \frac{\delta'}{-b\theta'+z_0}\cdot \frac{J_{F'}}{(-b \theta'+z_0)^{2n-1}} = \delta\cdot J_F,
\end{equation}
where in the last step above, we used the fact that
$$J_F = \phi_{2n-1,\gamma}(J_{F'}) = (-b \theta'+z_0)^{-(2n-1)} \cdot J_{F'}.$$
Notice that $\on{N}(I)^2 = \on{N}(\delta) \cdot \on{N}(J_{F})$ (this can be verified directly from~\eqref{eq-normofI} and~\eqref{eq-inspect}, but it also follows from the fact that $\on{N}(I')^2 = \on{N}(\delta') \cdot \on{N}(J_{F'})$ by Lemma~\ref{lem-idprime}).
\begin{lemma} \label{lem-welldefined}
The fractional ideal $I$ is well-defined, in the sense that it does not depend on the choice of $\gamma = \left[\begin{smallmatrix} z_0 & b\\ -x_0 & d\end{smallmatrix}\right]\in \on{SL}_2(R)$ satisfying $(x_0, z_0) \cdot \gamma = (0,1)$ and $d + \theta b \in K_F^\times$.
\end{lemma}
\begin{proof}
Let $\gamma = \left[\begin{smallmatrix} z_0 & b\\ -x_0 & d\end{smallmatrix}\right]\in \on{SL}_2(R)$ satisfying $(x_0, z_0) \cdot \gamma = (0,1)$ and $d + \theta b \in K_F^\times$ as above. Any other $\check{\gamma} \in \on{SL}_2(R)$ satisfying $(x_0,z_0) \cdot \check{\gamma}$ to $(0,1)$ is a right-translate of $\gamma$ by an element of the stabilizer of $(0,1)$, which is the group of unipotent matrices $M_t = \left[\begin{smallmatrix} 1 & t \\ 0 & 1 \end{smallmatrix} \right]$ for $t \in R$. So take $t \in R$, and let
$$\check{\gamma} = \gamma \cdot M_t = \left[\begin{array}{cc} z_0  & b + t z_0 \\ -x_0 & d - t x_0 \end{array} \right]$$
Let $\check{F}(x,z) = F((x,z) \cdot \check{\gamma}^{-1}) = \sum_{i = 0}^{2n+1} \check{f}_i x^{2n+1-i}z^i$, and let $\check{\theta}$ be the root of $\check{F}(x,1)$ in $K_F$. Assume that $t$ has been chosen so that $(d-t x_0) + \theta  (b+tz_0) \in K_F^\times$, and notice that
\begin{equation} \label{eq-checktheta}
\check{\theta} = \frac{-x_0 + \theta z_0}{(d-t x_0) + \theta  (b+tz_0)} = \frac{\theta'}{1 + \theta' t}.
\end{equation}
Just as we associated a pair $(I', \delta')$ to $\gamma$, we can associate a pair $(\check{I}, \check{\delta})$ to $\check{\gamma}$ satisfying the analogous properties $\check{I}^2 \subset \check{\delta} \cdot J_{\check{F}}$ and $\on{N}(\check{I})^2 = \on{N}(\check{\delta}) \cdot \on{N}(J_{\check{F}})$. To prove the lemma, it suffices to show that
\begin{equation} \label{eq-neededequality}
\phi_{n,\gamma}(I') = \phi_{n,\check{\gamma}}(\check{I}) \quad \text{and} \quad \phi_{1,\gamma}(\delta') = \phi_{1, \check{\gamma}}(\check{\delta}),
\end{equation}
The second equality in~\eqref{eq-neededequality} is readily deduced by inspecting~\eqref{eq-inspect}; as for the first equality, notice that it is equivalent to the following:
\begin{equation} \label{eq-desireewelldef}
\left(\frac{-b \theta'+z_0}{-(b+tz_0) \check{\theta}+ z_0}\right)^n \cdot \check{I} = I'.
\end{equation}
Using~\eqref{eq-checktheta}, we find that the fraction in parentheses in~\eqref{eq-desireewelldef} is equal to $1+t \theta'$. Now, observe that we can express the ideals $I'$ and $\check{I}$ as
\begin{equation} \label{eq-reexp}
I' = R \langle \wt{\xi}, \theta' \cdot I_{F'}^{n-1} \rangle \quad \text{and} \quad \check{I} = R \langle \wt{\xi}, \check{\theta} \cdot I_{\check{F}}^{n-1}\rangle,
\end{equation}
 where $\wt{\xi} = y_0$ if $y_0 \neq 0$ and $\wt{\xi} = f_{2n}' = \check{f}_{2n}$ if $y_0 = 0$. Upon combining~\eqref{eq-desireewelldef} and~\eqref{eq-reexp}, we find that
\begin{align*}
(1 + t \theta')^n \cdot \wt{I} & = R\langle (1+t \theta')^n \cdot \wt{\xi} , ((1+t\theta') \wt{\theta}) \cdot ((1+t\theta')I_{\wt{F}})^{n-1}\rangle \\
& = R\langle(1+t\theta')^n \cdot \wt{\xi}, \theta' \cdot I_{F'}^{n-1} \rangle = I',
\end{align*}
where the last step above follows because ${\theta'}^i \in I'$ for each $i \in \{0, \dots, n\}$.
\end{proof}
By applying the correspondence in Theorem~\ref{thm-theconstruction} to the pair $(I, \delta)$ constructed above, we see that the point $(x_0, y_0, z_0) \in \mathscr{S}_F(R)$ gives rise to a pair $(A,B) \in R^2 \otimes_R \Sym_2 R^{2n+1}$ with invariant form $(-1)^n \cdot F_{\on{mon}}$. To complete the construction, it remains to verify that $A$ is split:
\begin{lemma} \label{lem-split}
The matrix $A$ defines a split quadratic form over $k$.
\end{lemma}
\begin{proof}
It suffices to exhibit an $n$-dimensional isotropic space $X$ over $k$ for the quadratic form defined by $A$. We claim that
$$X = k\langle x_0 - \theta z_0, \theta(x_0 - \theta z_0), \dots, \theta^{n-1}(x_0 - \theta z_0)\rangle \subset K_F$$
does the job. Indeed, since $F$ is separable, the elements $\theta^i(x_0 - \theta z_0)$ for $i \in \{0, \dots, n-1\}$ are linearly independent over $k$: if $y_0 \neq 0$, then this follows because the elements $\theta^i \in K_F$ for $i \in \{0, \dots, n-1\}$ are linearly independent and $x_0 - \theta z_0$ is an invertible element of $K_F$; if $y_0 = 0$, then this follows because the elements $\theta^i \in K_{\wt{F}}$ for $i \in \{0, \dots, n-1\}$ are linearly independent and $x_0 - \theta z_0$ is an invertible element of $K_{\wt{F}}$. Thus, we have that $\dim_k X = n$. Moreover, notice that
$$\pi_{2n}(\langle \theta^i(x_0 - \theta z_0) , \theta^j (x_0 - \theta z_0) \rangle) = \pi_{2n}(x_0\theta^{i+j} - z_0\theta^{i+j+1}) = 0$$
for any $i,j \in \{0, \dots, n-1\}$.
\end{proof}

\begin{remark} \label{rmk-thefix}
We defined the pair $(I,\delta)$ in piecewise fashion, according as $y_0 \neq 0$ or $y_0 = 0$. The resulting pair $(A,B)$ nevertheless admits a ``functorial'' definition: indeed, the entries of $A$ and $B$ are polynomial functions in $x_0, y_0, z_0, b, d$ and the coefficients of $F$. Thus, if our purpose was simply to construct orbits from integral points, we could have carried out the construction for $y_0 \neq 0$ and specialized the resulting pair of matrices to the case $y_0 = 0$. However, our proofs of Theorems~\ref{thm-main} and~\ref{thm-hasse} require knowledge of the element $\delta \in K_F^\times$ corresponding to the orbit arising from an integral point, and it is difficult to reverse-engineer $\delta$ from a representative $(A,B)$ of such an orbit.

In~\cite[\S2]{thesource}, Bhargava establishes an analogous construction for rational points on hyperelliptic curves, one that takes a \emph{non}-Weierstrass point on a degree-$2n$ hyperelliptic curve $y^2 = F(x,z)$ and produces a pair $(I, \delta)$ of a fractional ideal $I$ and $\delta \in K_F^\times$ such that $I^2 \subset \delta \cdot I_F^{2n-3}$ and $\on{N}(I)^2 = \on{N}(\delta) \cdot \on{N}(I_F^{2n-3})$. By a parametrization due to Wood (see~\cite[Theorem~5.7]{MR3187931}), the pair $(I,\delta)$ gives rise to a pair $(A,B) \in R^2 \otimes_R \Sym_2 R^{2n}$ such that $\det(x \cdot A + z \cdot B) = (-1)^n \cdot F(x,z)$. To handle Weierstrass points, Bhargava uses the aforementioned specialization trick but does not determine the element $\delta \in K_F^\times$ corresponding to the specialized orbit, even though this data seems necessary for the proofs of~\cite[Theorem~3 and Corollary~4]{thesource} to go through.\footnote{While $100\%$ of hyperelliptic curves have no $\BQ$-rational Weierstrass points, they often have Weierstrass points over completions of $\BQ$, and such local Weierstrass points cannot be disregarded in proving~\cite[Theorem~3 and Corollary~4]{thesource}, which concern locally soluble $2$-covers (and hence local rational points) of hyperelliptic curves.} We note here that the piecewise construction of the pair $(I,\delta)$ given in this section can be modified to work for hyperelliptic curves, thus yielding complete proofs of~\cite[Theorem~3 and Corollary~4]{thesource}.
\end{remark}

\begin{remark} \label{rem-robust}
That the invariant form of the pair $(A,B)$ is equal to $(-1)^n \cdot F_{\on{mon}}$, that $A$ defines a split quadratic form, and that the stabilizer of $(A,B)$ contains a subgroup isomorphic to $R_F^\times[2]$ are properties that hold \emph{formally}. Thus, if we replace $R$ with any base ring and $F \in R[x,z]$ is \emph{any} binary form of degree $2n+1$, \emph{any} $R$-primitive solution $(x,y,z) = (x_0,y_0,z_0) \in R^3$ to the superelliptic equation $y^2 = F(x,z)$ gives rise to a pair of matrices $(A,B) \in R^2 \otimes_R \Sym_2 R^{2n+1}$ such that the invariant form is $F_{\on{mon}}$ and such that the stabilizer in $\on{SL}_{2n+1}(R)$ contains a subgroup isomorphic to $R_F^\times[2]_{\on{N}\equiv1}$. 
If the base ring $R$ is an integral domain, then $A$ defines a split quadratic form over the fraction field $k$ of $R$. Note that these properties hold even if we do not assume that $f_0 \neq 0$, that $F$ is separable, and that $k$ has characteristic $0$, despite the fact that we used these assumptions to derive the construction.
\end{remark}

\subsection{Distinguished Points} \label{sec-whosdist}

We say that a point $(x_0, y_0, z_0) \in \mathscr{S}_F(R)$ is \emph{distinguished} if its associated $G(k)$-orbit is distinguished. Note that Proposition~\ref{prop-rationalcorresp} implies that $(x_0, y_0, z_0)$ is distinguished if and only if $f_0 \cdot (x_0 - \theta z_0)$ is a square in $K_F^\times$.

When $f_0$ is a multiple of a perfect square by a unit $u \in R^\times$, $\mathscr{S}_F$ has a trivial $R$-point, namely $(u,\sqrt{u^{2n+1} \cdot f_0},0)$. Since $f_0\cdot (u - \theta \cdot 0) = u \cdot f_0 \in K_F^{\times 2}$, the orbits associated to these trivial solutions are always distinguished. Whether a primitive integer solution is distinguished is \emph{not necessarily} invariant under replacing $F$ with an $\on{SL}_2(R)$-translate. Indeed, suppose $(x_0, y_0, z_0) \in \mathscr{S}_F(R)$, take $\gamma \in \on{SL}_2(R)$ such that $(x_0, z_0) \cdot \gamma = (1,0)$, and let $F'(x,z) = F((x,z) \cdot \gamma^{-1})$. Then the point $(1,y_0,0) \in \mathscr{S}_{F'}(R)$ is distinguished, regardless of whether the  point $(x_0,y_0,z_0) \in \mathscr{S}_F(R)$ is distinguished. Nevertheless, whether a solution is distinguished is \emph{always} preserved under replacing $F$ with a translate by a unipotent matrix of the form $M_t^T$ for $t \in R$.

Now take $R = \BZ$. When $|f_0|$ is not a perfect square, we show in~\cite[\S3.3]{Swpreprint} that the non-square factors of $|f_0|$ tend to obstruct the distinguished $G(\BQ)$-orbit from arising via the map $\mathsf{orb}_F$ defined in \S\ref{sec-param}. More precisely, we have the following result, which states that there are strong limitations on the forms $F \in \mathscr{F}_{2n+1}(f_0)$ for which some point of $\scr{S}_F(\BZ)$ is distinguished:
\begin{theorem} \label{thm-caseprime}
Let $F \in \mathscr{F}_{2n+1}(f_0)$ be separable, and suppose that there exists a point of $\scr{S}_F(\BZ)$ giving rise to the distinguished $G(\BQ)$-orbit via the construction in \S\ref{sec-buildabear}. Then one of the following two mutually exclusive possibilities holds for each prime $p \mid \upkappa$:
\begin{enumerate}[leftmargin=2em]
\item[$(\mathrm{a})$] $R_F \otimes_\BZ \BZ_p$ is not the maximal order in $K_F \otimes_\BQ \BQ_p$; or
\item[$(\mathrm{b})$] $R_F \otimes_\BZ \BZ_p$ is the maximal order in $K_F \otimes_\BQ \BQ_p$ and $F/z$ is a unit multiple of a square modulo $p$.
\end{enumerate}
Moreover, the density of binary forms $F \in \mathscr{F}_{2n+1}(f_0)$ such that one of the conditions $(\mathrm{a})$ or $(\mathrm{b})$ holds for every prime $p \mid \upkappa$ is equal to $\mu_{f_0} \defeq \prod_{p \mid \upkappa}\left(\frac{1}{p^2} + \frac{p-1}{p^{n+1}}\right)$.
\end{theorem}
\begin{proof}
Let $p \mid \upkappa$. By~\cite[Proposition~40]{Swpreprint}, if $F \in \mathscr{F}_{2n+1}(f_0)$ is such that some element of the image of $\mathsf{orb}_F$ has distinguished $G(\BQ_p)$-orbit, then one of conditions (a) or (b) in the theorem statement holds. By~\cite[Proposition~3.5]{MR2367325}, the $p$-adic density of $F \in \mathscr{F}_{2n+1}(f_0,\BZ_p)$ satisfying condition (a) is $p^{-2}$, and it follows from~\cite[Proposition~41]{Swpreprint} that the $p$-adic density of $F\in \mathscr{F}_{2n+1}(f_0,\BZ_p)$ satisfying condition (b) is $p^{-n-1}(p-1)$. The theorem follows by combining these density calculations.
\end{proof}

Theorem~\ref{thm-caseprime} implies that points of $\scr{S}_F(\BZ)$ often fail to be distinguished --- a fact that is crucial for the proofs of Theorems~\ref{thm-main} and~\ref{thm-hasse}, because the geometry-of-numbers results from~\cite[\S9]{MR3156850} that we apply in \S\ref{sec-finally} to count orbits of $G(\BZ)$ on $V(\BZ)$ only give us a count of the non-distinguished orbits.

\section{Proof of Theorem~\ref{thm-main}} \label{sec-finally}

The purpose of this section is to give a proof of our first main result, namely Theorem~\ref{thm-main}.

\subsection{Outline of the Proof}

Let $\delta$ be the upper density of forms $F \in \mathscr{F}_{2n+1}(f_0)$, enumerated by height, such that $\mathscr{S}_F(\BZ) \neq \varnothing$. Proving Theorem~\ref{thm-main} amounts to obtaining an upper bound on $\delta$. 

To this end, let $\mathscr{F}_{2n+1}^*(f_0)$ be the subset of all $F \in \mathscr{F}_{2n+1}(f_0)$ satisfying both conditions in Theorem~\ref{thm-caseprime} for every prime $p \mid \upkappa$; note that the complement of $\mathscr{F}_{2n+1}^*(f_0)$ consists entirely of binary forms for which the distinguished orbit does not arise from a primitive integer solution. Moreover, let $\mathscr{P}$ denote the ($G(\BZ)$-invariant) subset of integral elements of $V(\BZ)$ that arise via the construction in \S\ref{sec-theconstrsec} from separable forms $F \in \mathscr{F}_{2n+1}(f_0)$. By partitioning $\mc{F}_{2n+1}(f_0)$ into $\mc{F}_{2n+1}^*(f_0)$ and its complement, we see that $\delta$ can be bounded as follows:
\begin{align} \label{eq-thekeybound0}
\delta & \leq \lim_{X \to \infty} \frac{\#\{F \in \mathscr{F}_{2n+1}^*(f_0) : H(F) < X\} + \#\big(G(\BQ)\backslash \{\text{$T \in \mathscr{P}$ : $T$ irreducible, $H(T) < X$}\}\big)}{\#\{F \in \mathscr{F}_{2n+1}(f_0) : H(F) < X\}},
\end{align}
where we say that $T$ is \emph{irreducible} if $\on{ch}(T)$ is separable and the $G(\BQ)$-orbit of $T$ is non-distinguished. 

The first term on the right-hand side of~\eqref{eq-thekeybound0} constitutes the trivial bound for the density of forms $F \in \mc{F}_{2n+1}^*(f_0)$ such that $\mathscr{S}_F(\BZ) \neq \varnothing$. Computing this first term is straightforward. Indeed, the first term in the numerator and the denominator are respectively given by
\begin{align}
\#\{F \in \mathscr{F}_{2n+1}^*(f_0) : H(F) < X\} & = \mu_{f_0} \cdot f_0^{-2n^2 - n} \cdot 2^{2n+1} \cdot X^{\frac{n+1}{2n}} + o(X^{\frac{n+1}{2n}}), \quad\text{and} \label{eq-numer1est}\\
\#\{F \in \mathscr{F}_{2n+1}(f_0) : H(F) < X\} & = f_0^{-2n^2 - n} \cdot 2^{2n+1} \cdot X^{\frac{n+1}{2n}} + o(X^{\frac{n+1}{2n}}). \label{eq-denomest}
\end{align}
The second term in the numerator on the right-hand side of~\eqref{eq-thekeybound0} means ``the number of $G(\BQ)$-equivalence classes of irreducible elements $T \in \mathscr{P}$ of height less than $X$,'' and this constitutes a bound on the number forms $F \in \mc{F}_{2n+1}(f_0) \smallsetminus \mc{F}_{2n+1}^*(f_0)$ such that $\mathscr{S}_F(\BZ) \neq \varnothing$. It is hard to describe $\mathscr{P}$ explicitly, and \mbox{hence to obtain} a precise asymptotic for this second term. But $\mathscr{P}$ is contained within a subset of $V(\BZ)$ defined by local conditions: indeed, we have $$\mathscr{P} \subset V(\BZ) \cap \bigcap_v \mathscr{P}_v,$$
where for each place $v$ of $\BQ$, we write $\mathscr{P}_v$ for the subset of $V(\BZ_v)$ that arises via the construction in \S\ref{sec-theconstrsec} from separable forms $F \in \mathscr{F}_{2n+1}(f_0, \BZ_v)$. Thus, the second term in the numerator on the right-hand side of~\eqref{eq-thekeybound0} is bounded by the following quantity:
\begin{align} \label{eq-thekeybound}
\#\bigg(G(\BQ)\bigg\backslash \bigg\{\text{$T \in V(\BZ) \cap \bigcap_v \mathscr{P}_v$ : $T$ is irreducible, $H(T) < X$}\bigg\}\bigg)
\end{align}
The rest of this section is devoted to bounding~\eqref{eq-thekeybound}. We desire a bound of order $o(f_0^{-2n^2-n} \cdot 2^{-n})$ when $2 \nmid f_0$ and a bound of order $O(f_0^{-2n^2-n} \cdot 2^{-\varepsilon_1 n^{\varepsilon_2}})$ when $2 \mid f_0$. We obtain bounds of these magnitudes by means of the following steps:
\begin{enumerate}
    \item[(A)] In \S\ref{sec-quotebharg}, we adapt results of Bhargava--Gross~\cite{MR3156850} to obtain an asymptotic upper bound for the number of irreducible orbits of $G(\BZ)$ on $V(\BZ)$ of bounded height, where the orbits are counted with respect to a weight function $\phi$ defined by infinitely many local conditions (see Theorem~\ref{thm-cong2}). By choosing the weight function $\phi$ appropriately --- so that it picks out orbits belonging to $\mathscr{P}_v$ for every place $v$ and counts each $G(\BQ)$-equivalence class of integral orbits just once --- we obtain an upper bound for the quantity~\eqref{eq-thekeybound}. This bound is expressed a product of local mass formulas, one for each place of $\BQ$.
    \item[(B)] Proving Theorem~\ref{thm-main} then amounts to bounding each of the local mass formulas obtained in Step (A) (in other words, we must control the ``size'' of $\mathscr{P}_v$ for each $v$). 
    \begin{itemize}
    \item In \S\ref{sec-overr}, we bound the mass at $\infty$ by working over $\BR$ (see Proposition~\ref{prop-thisisforr}). The mass at $\infty$ contributes a factor of $O(2^n \cdot X^{\frac{n+1}{2n}})$ to the bound. 
    \item In \S\ref{sec-overp}, we bound the mass at $p$ when $2 \neq p \nmid f_0$ by working over $\BZ/p\BZ$ (see Proposition~\ref{prop-denstoobig}). As we later show in \S\ref{sec-hillary}, the masses at these primes together contribute a factor of $O(2^{-\varepsilon_1 n^{\varepsilon_2}})$ to the bound. This constitutes a small part of the saving when $2 \nmid f_0$ and the entire saving when $2 \mid f_0$.
    \item In \S\ref{sec-orbates}, we bound the mass at $p$ when $2 = p \nmid f_0$ by working over $\BZ/8\BZ$ (see Proposition~\ref{prop-2denstoobig}). In this case, the mass at $2$ contributes a factor of $O(2^{-2n})$ to the bound, which constitutes most of the saving.
    \item In \S\ref{sec-overzpqp}, we bound the mass at $p$ when $p \mid f_0$ by working over $\BQ_p$ and $\BZ_p$ (see Propositions~\ref{prop-qpbound} and~\ref{prop-zpbound}). In this case, the mass at $p$ contributes to the bound a factor of $O(|f_0|_p^{2n^2+n})$ when $p \neq 2$ and $O(2^{-n} \cdot |f_0|_2^{2n^2+n})$ when $p = 2$.
    \end{itemize}
    \item[(C)] The final step is to combine the bounds obtained in Steps (A) and (B) and to simplify the result. We do this in \S\ref{sec-hillary}.
\end{enumerate}

\subsection{Step (A): Counting Irreducible Integral Orbits of Bounded Height} \label{sec-quotebharg}

For a monic polynomial $f = x^{2n+1} + \sum_{i = 1}^{2n+1} c_i \cdot x^{2n+1-i} \in \BZ[x]$, let the \emph{height} of $f$ be defined by $H(f) \defeq H(c_1, \dots, c_{2n+1}) \defeq \max\{|c_i|^{2n(2n+1)/i} : i = 1, \dots, 2n+1\}$. For an element $T \in V(\BZ)$, let the \emph{height} of (the $G(\BZ)$-orbit of) $T$ be defined by $H(T) = H(\on{ch}(T))$. 

Fix $m \in \{0, \dots, n\}$, and define $\mc{I}(m) \subset \scr{F}_{2n+1}(1,\BR)$
to be the set of separable monic degree-$(2n+1)$ polynomials over $\BR$ having exactly $2m+1$ real roots. Let $\mc{B} \subset \on{ch}^{-1}(\mathscr{I}(m)) \subset V(\BR)$ be a measurable subset closed under the action of $G(\BR)$. Let $N(V(\BZ) \cap \mc{B}; X)$ denote the number of $G(\BZ)$-orbits of irreducible elements $T \in V(\BZ) \cap \mc{B}$ satisfying $H(T) < X$. Then we have the following adaptation of~\cite[Theorem~10.1~and~(10.27)]{MR3156850}, giving an asymptotic for $N(V(\BZ) \cap \mc{B}; X)$:
\begin{theorem} \label{thm-thecount}
There exists an absolute constant $\mc{J} \in \BQ^\times$ such that
$$N(V(\BZ) \cap \mc{B}; X) = \frac{1}{2^{m+n}} \cdot |\mc{J}| \cdot \on{Vol}(G(\BZ) \backslash G(\BR)) \cdot  \int_{\substack{f \in \mc{I}(m) \\ H(f) < X}} \#\left(\frac{\mc{B} \cap \on{ch}^{-1}(f)}{G(\BR)}\right) df + o(X^{\frac{n+1}{2n}}),$$
where the fundamental volume $\on{Vol}(G(\BZ) \backslash G(\BR))$ is computed with respect to a differential
$\omega$ that generates the rank-$1$ module of top-degree differentials of $G$ over $\BZ$.
\end{theorem}
\begin{proof}
In~\cite{MR3156850}, Bhargava and Gross work with the subrepresentation $V' \subset V$ consisting of the traceless operators in $V$. Such traceless operators have characteristic polynomial with the coefficient $c_1$ of the term $x^{2n}$ equal to $0$. They prove that $N(V'(\BZ) \cap \mc{B}; X)$ is given by a formula (see~\cite[Theorem~10.1~and~(10.27)]{MR3156850}) that is almost identical to the formula in the statement of Theorem~\ref{thm-thecount}, except that the value of $\mc{J}$ may be different and the integral runs over polynomials with coefficients $(0, c_2, \dots, c_{2n+1})$ in $\mc{I}(m)$. By simply lifting the restriction that $c_1 = 0$, it is not hard to check that the proof of~\cite[Theorem~10.1~and~(10.27)]{MR3156850} carries over with minimal modifications to prove Theorem~\ref{thm-thecount}.
\end{proof}

The following result generalizes Theorem~\ref{thm-thecount} by giving an upper bound on the number of integral orbits of height up to $X$, where the orbits are weighted by a function defined by local conditions:
\begin{theorem}[\protect{cf.~\cite[Theorem~10.12]{MR3156850}}] \label{thm-cong2}
Let $\phi \colon V(\BZ) \to [0,1] \subset \BR$ be a function such that there exists a function $\phi_p \colon V(\BZ_p) \to [0,1]$ for each prime $p$ satisfying the following conditions:
\begin{itemize}[leftmargin=2em]
    \item For all $T \in V(\BZ)$, the product $\prod_{p} \phi_p(T)$ converges to $\phi(T)$; and
    \item For each $p$, the function $\phi_p$ is locally constant outside a closed set $S_p \subset V(\BZ_p)$ of measure $0$.
\end{itemize}
Let $N_\phi(V(\BZ) \cap \mc{B}; X)$ denote the number of $G(\BZ)$-orbits of irreducible elements $T \in V(\BZ) \cap \mc{B}$ satisfying $H(T) < X$, where each orbit $G(\BZ) \cdot T$ is counted with weight \mbox{$\phi(T) \defeq \prod_p \phi_p(T)$. Then}
$$N_\phi(V(\BZ) \cap \mc{B}; X) \leq N(V(\BZ) \cap \mc{B}; X) \cdot \prod_{p} \int_{T \in V(\BZ_p)} \phi_p(T)dT + o(X^{\frac{n+1}{2n}}).$$
\end{theorem}
\begin{proof}
The proof is identical to the first half of the proof of~\cite[Theorem~2.21]{MR3272925} (note: the ``accep-table'' hypothesis in~\cite[Theorem~2.21]{MR3272925} is not needed for the upper bound in Theorem~\ref{thm-cong2}).
\end{proof}

To apply Theorem~\ref{thm-cong2} to bound~\eqref{eq-thekeybound}, we need to choose the local weight functions $\phi_p$ appropriately. Given $T \in V(\BZ)$, denote by $m(T)$ the 
quantity
$$m(T) \defeq \sum_{T' \in \mc{O}(T)} \frac{\#\on{Stab}_\BQ(T')}{\#\on{Stab}_\BZ(T')} = \sum_{T' \in \mc{O}(T)} \frac{\#\on{Stab}_\BQ(T)}{\#\on{Stab}_\BZ(T')},$$
where $\mc{O}(T)$ denotes a set of representatives for the action of $G(\BZ)$ on the $G(\BQ)$-orbit of $T$, and where $\on{Stab}_\BQ(T')$ and $\on{Stab}_\BZ(T')$ respectively denote the stabilizers of $T'$ in $G(\BQ)$ and $G(\BZ)$. By~\cite[Proposition~10.8]{MR3156850}, counting $G(\BQ)$-orbits on $V(\BZ)$ is, up to a negligible error, no different from counting $G(\BZ)$-orbits on $V(\BZ)$, where the $G(\BZ)$-orbit of $T \in V(\BZ)$ is weighted by $1/m(T)$.

The weight $m(T)$ can be expressed as a product over primes $p$ of local weights $m_p(T)$. Indeed, as stated in~\cite[(11.3)]{MR3156850}, we have that
\begin{equation} \label{eq-weightproduct}
m(T) = \prod_{p} m_p(T), \quad \text{where} \quad m_p(T) \defeq \sum_{T' \in \mc{O}_p(T)} \frac{\#\on{Stab}_{\BQ_p}(T)}{\#\on{Stab}_{\BZ_p}(T')},
\end{equation}
where $\mc{O}_p(T)$ denotes a set of representatives for the action of $G(\BZ_p)$ on the $G(\BQ_p)$-orbit of $T$, and where $\on{Stab}_{\BQ_p}(T')$ and $\on{Stab}_{\BZ_p}(T')$ respectively denote the stabilizers of $T'$ in $G(\BQ_p)$ and $G(\BZ_p)$.

For each prime $p$, let $\psi_p \colon V(\BZ_p) \to \{0,1\} \subset \BR$ be the indicator function of elements $T \in \scr{P}_p$ such that $\on{ch}(T)$ is separable. Then upon taking $\phi_p = \psi_p/m_p$ for each $p$, Theorems~\ref{thm-thecount} and~\ref{thm-cong2} together imply the following:
\begin{proposition}
We have that~\eqref{eq-thekeybound} is bounded above by
\begin{align} \label{eq-deltamiddle}
    &  \sum_{m = 0}^n \frac{1}{2^{m+n}} \cdot |\mc{J}| \cdot \on{Vol}(G(\BZ) \backslash G(\BR))  \cdot\int_{\substack{f \in \mc{I}(m) \\ H(f) < X}} \#\left(\frac{\mc{P}_\infty \cap \on{ch}^{-1}(f)}{G(\BR)}\right) df \cdot \\
    & \qquad\qquad\qquad\qquad\qquad\qquad\qquad\qquad\qquad\qquad\qquad\qquad \prod_{p} \int_{T \in V(\BZ_p)} \frac{\psi_p(T)}{m_p(T)}dT + o(X^{\frac{n+1}{2n}}).  \nonumber
\end{align}
\end{proposition}
In what follows, we call the sum over $m$ on the first line of~\eqref{eq-deltamiddle} the \emph{mass at $\infty$}, and we call the integral over $V(\BZ_p)$ on the second line of~\eqref{eq-deltamiddle} the \emph{mass at $p$}.

\subsection{Step (B): Bounding the Mass at $\infty$} \label{sec-overr}

In this section, we prove the following result, which shows that the mass at $\infty$ in~\eqref{eq-deltamiddle} is bounded by $O(2^n \cdot X^{\frac{n+1}{2n}})$:
\begin{proposition} \label{prop-thisisforr}
    The mass at $\infty$ in~\eqref{eq-deltamiddle} is
    $$\leq 3 \cdot 2^{n} \cdot \prod_{p > 2} \frac{p^{2n^2+n}}{\#G(\BZ/p\BZ)} \cdot \int_{H(f) < X} df.$$
\end{proposition}
\begin{proof}
We start with the following lemma, where we derive an upper bound on the factor $|\mc{J}| \cdot \on{Vol}(G(\BZ) \backslash G(\BR))$ that occurs in the mass at $\infty$:
\begin{lemma} \label{lem-compJ}
We have that $$|\mc{J}|\cdot \on{Vol}(G(\BZ) \backslash G(\BR))   \leq 2^{2n+1} \cdot \prod_{p > 2} \frac{p^{2n^2+n}}{\#G(\BZ/p\BZ)}.$$
\end{lemma}
\begin{proof}
 Let $R = \BC$, $\BR$, or $\BZ_p$ where $p$ is a prime. Let $\mc{R} \subset R^{2n+1}$ be any open subset, and let $s \colon \mc{R} \to V(R)$ be a continuous function such that the characteristic polynomial of $s(c_1, \dots, c_{2n+1})$ is given by $x^{2n+1} + \sum_{i = 1}^{2n+1} c_ix^{2n+1-i}$ for every $(c_1, \dots, c_{2n+1}) \in \mc{R}$. The constant $\mc{J}$ arises as a multiplicative factor in the following change-of-measure formula:
 \begin{lemma} \label{prop-defJ}
For any measurable function $\phi$ on $V(R)$, we have
 $$\int_{T \in G(R) \cdot s(\mc{R})} \phi(T) dT = |\mc{J}| \cdot \int_\mc{R} \int_{G(R)} \phi(g \cdot s(c_1, \dots, c_{2n+1})) \omega(g) dr,$$
 where we regard $G(R) \cdot s(\mc{R})$ as a multiset, $|-|$ denotes the standard absolute value on $R$, $dT$ is the Euclidean measure on $V(R)$, and $dr$ is the restriction to $\mc{R}$ of the Euclidean measure on $R^{2n+1}$.
 \end{lemma}
 \begin{proof}[Proof of Lemma~\ref{prop-defJ}]
 The proof is identical to that of~\cite[Proposition~3.11]{MR3272925}.
 \end{proof}
 In particular, the constant $\mc{J}$ is independent of the choices of $R = \BC$, $\BR$, or $\BZ_p$ and of the region $\mc{R}$ and the functions $s$ and $\phi$. Thus, to compute $|\mc{J}|$, we can make the following convenient choices: take $R = \BZ_p$, so that $|-| = |-|_p$; take
 $$\mc{R} = \left\{(c_1, \dots, c_{2n+1}) \in \BZ_p^{2n+1} : x^{2n+1} + \sum_{i = 1}^{2n+1} c_i x^{2n+1-i} \equiv f \pmod p\right\}$$
for a fixed monic irreducible degree-$(2n+1)$ polynomial $f \in (\BZ/p\BZ)[x]$; letting $\Sigma_f \defeq \{T \in V(\BZ_p) : \on{ch}(T) \equiv f \pmod p\}$, take $s$ to be any continuous right-inverse to the function that sends $T \in \Sigma_f$ to the list of coefficients of $\on{ch}(T)$; and take $\phi$ to be the function that sends $T \in \Sigma_f$ to $\frac{1}{\#\on{Stab}(T)}$, where $\on{Stab}(T) \subset G(\BZ_p)$ is the stabilizer of $T$, and sends $T \in V(\BZ_p) \smallsetminus \Sigma_f$ to $0$. The existence of the right-inverse $s$ is well-known; see, e.g.,~\cite[\S3.3]{cuspy}.

Because $\phi$ is $G(\BZ_p)$-invariant, Lemma~\ref{prop-defJ} yields that for the above convenient choices, we have, on the one hand, that the $p$-adic density of $\Sigma_f$ is:
 \begin{align}
     \on{Vol}(\Sigma_f) \defeq \int_{T \in \Sigma_f} dT & = \int_{T \in G(\BZ_p) \cdot s(\mc{R})} \frac{1}{\#\on{Stab}(T)} dT \nonumber \\
     & = |\mc{J}|_p \cdot \on{Vol}(G(\BZ_p)) \cdot \int_{f'\equiv f\,(\on{mod} p)}\sum_{\substack{T \in G(\BZ_p)\backslash\Sigma_f \\ \on{ch}(T) = f'}} \frac{1}{\#\on{Stab}(T)} dr. \label{eq-Jintexp}
     \end{align}
      First suppose $p \neq 2$. We now compute the sum in the integrand in~\eqref{eq-Jintexp}. Because we chose $f$ to be irreducible over $\BZ/p\BZ$, the ring $R_{f'} \defeq \BZ_p[x]/(f')$ is the maximal order in its field of fractions $K_{f'}$, and it is in fact the unique local ring of rank $2n+1$ over $\BZ_p$ having residue field $\mathbb{F}_{p^{2n+1}}$ (hence $R_{f'}$ does not depend on the choice of $f'$). Then by~\cite[(63)]{Swpreprint}, there is precisely one $G(\BZ_p)$-orbit with characteristic polynomial $f'$, and the size of the stabilizer of this orbit is equal to $\#R_{f'}^\times[2]_{\on{N}\equiv1} = 1$. Thus, we find that
     \begin{equation}
      \sum_{\substack{T \in G(\BZ_p)\backslash\Sigma_f \\ \on{ch}(T) = f'}} \frac{1}{\#\on{Stab}(T)} =  1. \label{eq-secondinvol}
      \end{equation}
Combining~\eqref{eq-Jintexp} and~\eqref{eq-secondinvol} yields that
\begin{equation}
      \on{Vol}(\Sigma_f) = |\mc{J}|_p \cdot \on{Vol}(G(\BZ_p)) \cdot \int_{f'\equiv f\,(\on{mod} p)} dr = |\mc{J}|_p \cdot \frac{\on{Vol}(G(\BZ_p))}{p^{2n+1}}. \label{eq-Jexp1}
     \end{equation}
     Let $\ol{\Sigma}_f \defeq \{T \in  V(\BZ/p\BZ) : \on{ch}(T) = f\}$. Then the mod-$p$ reduction map $\Sigma_f \to \ol{\Sigma}_f$ is surjective, and we have by~\cite[\S6.1]{MR3156850} that $\#\ol{\Sigma}_f = \#G(\BZ/p\BZ)$.
     Thus, we have on the other hand that
\begin{equation}
\on{Vol}(\Sigma_f) = \frac{\#\ol{\Sigma}_f}{p^{\dim V}} = \frac{\#G(\BZ/p\BZ)}{p^{2n^2+3n+1}}.      \label{eq-Jexp2}
 \end{equation}
 Equating~\eqref{eq-Jexp1} and~\eqref{eq-Jexp2} yields that
 \begin{equation} \label{eq-podd}
 |\mc{J}|_p =  \frac{\#G(\BZ/p\BZ)}{p^{2n^2+n} \cdot \on{Vol}(G(\BZ_p))} = 1
 \end{equation}
 when $p \neq 2$, where the last equality follows from the fact that the group $G$ is smooth over $\BZ_p$ when $p \neq 2$.

 Next, suppose $p = 2$. We now derive an upper bound on the sum in the integrand in~\eqref{eq-Jintexp}. By~\cite[Proposition~59]{swathesis}, the number of $G(\BZ_2)$-orbits with characteristic polynomial $f'$ is equal to $2^{n-1}(2^n+1)$.
  The size of the stabilizer of any such orbit is equal to $\#R_{f'}^\times[2]_{\on{N}\equiv 1} = 1$, so we find that
 \begin{equation} \label{eq-secondinvol2}
 \sum_{\substack{T \in G(\BZ_p)\backslash\Sigma_f \\ \on{ch}(T) = f'}} \frac{1}{\#\on{Stab}(T)} = 2^{n-1}(2^n+1).
 \end{equation}
 Combining~\eqref{eq-Jintexp} and~\eqref{eq-secondinvol2} yields that
 \begin{align} \label{eq-comparizon1}
    \on{Vol}(\Sigma_f) & = |\mc{J}|_2 \cdot \on{Vol}(G(\BZ_2)) \cdot\frac{2^{n-1}(2^n+1)}{2^{2n+1}}.
 \end{align}
 Let $\ol{\Sigma}_f = \{T \in V(\BZ/2\BZ) : \on{ch}(T) = f\}$. Then the mod-$2$ reduction map $\Sigma_f \to \ol{\Sigma}_f$ is surjective,  and by Propositions~\ref{prop-slo} and~\ref{prop-rationalcorresp}, the action of $G(\BZ/2\BZ)$ on $\ol{\Sigma}_f$ is simply transitive, so $\#\ol{\Sigma}_f = \#G(\BZ/2\BZ)$. Thus, we have on the other hand that
 \begin{equation} \label{eq-comparizon2}
\on{Vol}(\Sigma_f) =  \frac{\#\ol{\Sigma}_f}{2^{\dim V}} = \frac{\#G(\BZ/2\BZ)}{2^{2n^2+3n+1}}.
 \end{equation}
       Combining~\eqref{eq-comparizon1} and~\eqref{eq-comparizon2} yields that
\begin{equation} \label{eq-neededlabel2}
|\mc{J}|_2^{-1} = 2^{2n} \cdot \on{Vol}(G(\mathbb{Z}_2)) \cdot \frac{2^{2n^2-1}(2^n+1)}{\#G(\BZ/2\BZ)}
\end{equation}
Since $G$ is not smooth over $\BZ_2$, computing $\on{Vol}(G(\BZ_2))$ is far more complicated, but we do not need to know the value of $\on{Vol}(G(\BZ_2))$ for our purpose. The value of $\#G(\BZ/2\BZ)$ is given in~\cite[\S6]{MR3347991} to be $\#G(\BZ/2\BZ) = 2^{n^2} \cdot \prod_{i = 1}^n (2^{2i}-1)$, so it follows that
\begin{align} \label{eq-jishardat2}
\frac{2^{2n^2-1}(2^n+1)}{\#G(\BZ/2\BZ)} & = (2^{-1} + 2^{-n-1}) \cdot \prod_{i = 1}^n (1 - 2^{-2i})^{-1} \leq 1.
\end{align}
Now, applying the identity
\begin{equation} \label{eq-normmultident}
|\mc{J}| \cdot \prod_{p} |\mc{J}|_p = 1
\end{equation}
along with the Tamagawa number identity (see~\cite{MR0213362})
\begin{equation} \label{eq-tam}
\on{Vol}(G(\BZ) \backslash G(\BR)) \cdot \prod_{p} \on{Vol}(G(\BZ_p)) = 2
\end{equation}
and combining them with~\eqref{eq-podd},~\eqref{eq-neededlabel2}, and~\eqref{eq-jishardat2} completes the proof of Lemma~\ref{lem-compJ}.
\end{proof}

Next, in the following lemma, we compute the integrand in the integral over $f$:
\begin{lemma} \label{prop-realbound}
 Let $F \in \mathscr{F}_{2n+1}(f_0,\BR)$, and suppose that $F$ has exactly $2m+1$ real roots. Then there are exactly $2m+1$ orbits of $G(\BR)$ on $V(\BR)$ that arise from points in $\mathscr{S}_F(\BR)$. In particular, the integrand in the integral in the mass at $\infty$ in~\eqref{eq-deltamiddle} is a constant with value $2m+1$.
\end{lemma}
\begin{proof}
 Let $f(x) = F_{\on{mon}}(x,1)$, and let the real roots of $f$ be given in increasing order by $\lambda_1 < \cdots < \lambda_{2m+1}$. We now compute the number of orbits of $G(\BR)$ on $V(\BR)$ having characteristic polynomial $f$ that arise from points in $\scr{S}_F(\BR)$ via the construction in \S\ref{sec-buildabear}. 
The $G(\BR)$-orbit associated to a point $(x_0, y_0, z_0) \in \mathscr{S}_F(\BR)$ is identified via the correspondence in Proposition~\ref{prop-rationalcorresp} with some $f_0 \cdot \delta \in (K_F^\times/K_F^{\times 2})_{\on{N} \equiv 1}$. If $y_0 \neq 0$, this class is represented by a sequence of the form
\begin{equation} \label{eq-sseq0}
(f_0(x_0 - \lambda_1 z_0), \dots, f_0(x_0 - \lambda_{2m+1} z_0), b_1, \dots, b_{n-m}) \in \BR^{2m+1} \times \BC^{n-m} \simeq K_F.
\end{equation}
Otherwise, if $y_0 = 0$, then $\frac{x_0}{z_0} = \lambda_j$ for some $j$; letting $\wt{F}$ be as in~\eqref{eq-defftilde}, we see that the class $f_0 \cdot \delta \in (K_F^\times/K_F^{\times 2})_{\on{N} \equiv 1}$ is represented by a sequence of the form
\begin{equation} \label{eq-sseq0.0}
(f_0(x_0 - \lambda_1 z_0), \dots,f_0 \cdot \wt{F}(x_0,z_0) ,\dots, f_0(x_0 - \lambda_{2m+1} z_0), b_1, \dots, b_{n-m}) \in \BR^{2m+1} \times \BC^{n-m},
\end{equation}
where the term $f_0 \cdot \wt{F}(x_0, z_0)$ in~\eqref{eq-sseq0.0} replaces the term $f_0(x_0 - \lambda_j z_0)$ in~\eqref{eq-sseq0}. The class in $(K_F^{\times}/K_F^{\times 2})_{\on{N} \equiv 1}$ of the sequence in~\eqref{eq-sseq0} or in~\eqref{eq-sseq0.0} is given by the sequence of signs of its first $2m+1$ terms. Because $\on{N}(b_j) > 0$ for every $j$, the condition that $\on{N}(f_0 \cdot \delta)$ is a square in $\BR^\times$ is equivalent to the condition that $\prod_{i = 1}^{2m+1} f_0(x_0 - \lambda_i z_0) > 0$ if $y_0 \neq 0$ and $(f_0 \cdot \wt{F}(x_0,z_0)) \cdot \prod_{\substack{i = 1 \\ i \neq j}}^{2m+1} f_0(x_0 - \lambda_i z_0) > 0$ if $y_0 = 0$. We now split into cases based on the signs of $f_0$ and $z_0$:

\vspace*{0.2cm}
\noindent \emph{Case 1a}: $f_0 > 0$, $z_0 > 0$. First suppose $y_0 \neq 0$. In this case, the sign of $f_0(x_0 - \lambda_i z_0)$ is equal to the sign of $(x_0/z_0) - \lambda_i$. The possible sequences of signs of the $f_0(x_0 - \lambda_i z_0)$ are therefore
\begin{equation} \label{eq-sseq1}
(+,-,-,-,\dots,-)\,\text{ or }\,(+,+,+,-,\dots,-)\,\text{ or }\, \dots\,\text{ or }\, (+,+,+,+,\dots,+),
\end{equation}
giving a total of $m+1$ orbits. We label the sign sequences in~\eqref{eq-sseq1} from left to right by an index $\tau$ that runs from $1$ up to $m+1$. If $y_0 = 0$ and $\frac{x_0}{z_0} = \lambda_j$, then because $\lambda_i < \lambda_j$ when $i < j$ and $\lambda_i > \lambda_j$ when $i > j$, we obtain the same sign sequences as those listed in~\eqref{eq-sseq1}.

\vspace*{0.2cm}
\noindent \emph{Case 1b}: $f_0 > 0$, $z_0 < 0$. First suppose $y_0 \neq 0$. In this case, the sign of $f_0(x_0 - \lambda_i z_0)$ is equal to the opposite of the sign of $(x_0/z_0) - \lambda_i$. The possible sequences of signs of the \mbox{$f_0(x_0 - \lambda_i z_0)$ are therefore}
\begin{equation} \label{eq-sseq2}
(+,+,+,\dots,+,+)\,\text{ or }\,(-,-,+,\dots,+,+)\,\text{ or }\, \dots\,\text{ or }\, (-,-,-,\dots,-,+),
\end{equation}
giving a total of $m+1$ orbits. We label the sign sequences in~\eqref{eq-sseq2} from left to right by an index $\tau$ that runs from $m+1$ up to $2m+1$. (Note that the first sign sequence in~\eqref{eq-sseq2} is the same as the last sign sequence in~\eqref{eq-sseq1}, which is why they share the same value of $\tau$.) If $y_0 = 0$, we again obtain the same sign sequences as those listed in~\eqref{eq-sseq2}.

\vspace*{0.2cm}
\noindent \emph{Case 2a}: $f_0 < 0$, $z_0 > 0$. The possible sign sequences are the same as those listed in~\eqref{eq-sseq2}, and we likewise label them from left to right by an index $\tau$ that runs from $m+1$ up to $2m+1$.

\vspace*{0.2cm}
\noindent \emph{Case 2b}: $f_0 < 0$, $z_0 < 0$. The possible sign sequences are the same as those listed in~\eqref{eq-sseq1}, and we likewise label them from left to right by an index $\tau$ that runs from $1$ up to $m+1$.
\end{proof}
%
Substituting the results of Lemmas~\ref{lem-compJ} and~\ref{prop-realbound} into the mass at $\infty$ in~\eqref{eq-deltamiddle}, we deduce that this mass is bounded by
\begin{equation}
    \leq \sum_{m = 0}^n \frac{1}{2^{m+n}} \cdot 2^{2n+1} \cdot \prod_{p > 2} \frac{p^{2n^2+n}}{\#G(\BZ/p\BZ)} \cdot \int_{\substack{f \in \scr{I}(m) \\ H(f) < X}} (2m+1) df \leq 3 \cdot 2^{n} \cdot \prod_{p > 2} \frac{p^{2n^2+n}}{\#G(\BZ/p\BZ)} \cdot \int_{H(f) < X} df,
\end{equation}
as necessary. This completes the proof of Proposition~\ref{prop-thisisforr}.
\end{proof}

\subsection{Step (B): Bounding the Mass at $p$ when $2 \neq p \nmid f_0$} \label{sec-overp}

Let $\mc{I}_p(m)$ denote the set of monic degree-$(2n+1)$ polynomials over $\BZ/p\BZ$ (\emph{not necessarily} separable) having $m$ distinct irreducible factors. When $2 \neq p \nmid f_0$, we obtain the following result, which bounds the mass at $p$ in~\eqref{eq-deltamiddle}: 
\begin{proposition} \label{prop-denstoobig}
For each prime $p \nmid 2f_0$, we have that
\begin{equation*}
   \int_{T \in V(\BZ_p)} \frac{\psi_p(T)}{m_p(T)} dT  \leq \min\left\{1,\sum_{m = 1}^{2n+1}  \frac{p+1}{2^{m-1}} \cdot \frac{\#G(\BZ/p\BZ)}{p^{2n^2+n}} \cdot \frac{\#\mc{I}_p(m)}{p^{2n+1}}\right\}
\end{equation*}
\end{proposition}
\begin{remark}
    As we later show in \S\ref{sec-hillary}, combining the bound in Proposition~\ref{prop-denstoobig} over all primes $p \nmid 2f_0$ yields a total saving of $O(2^{-\varepsilon_1 n^{\varepsilon_2}})$.
\end{remark}
\begin{proof}[Proof of Proposition~\ref{prop-denstoobig}]

 Given a monic degree-$(2n+1)$ polynomial $f$ over $\BZ/p\BZ$, let $\ol{\Sigma}_f^{\on{sol}} \subset V(\BZ/p\BZ)$ denote the mod-$p$ reduction of the subset $\{T \in \mathscr{P}_p : \on{ch}(T) \equiv f \pmod p\}$. Because $G$ is smooth over $\BZ_p$, the map $G(\BZ_p) \to G(\BZ/p\BZ)$ is surjective, implying that $\ol{\Sigma}_f^{\on{sol}}$ is $G(\BZ/p\BZ)$-invariant. Then, since $m_p(T) \geq 1$, we have
\begin{align}
 \int_{T \in V(\BZ_p)} \frac{\psi_p(T)}{m_p(T)} dT &\leq \int_{T \in \mathscr{P}_p} dT  \label{eq-porbit1} \\
 & \leq \sum_{m = 1}^{2n+1} \sum_{f \in \mc{I}_p(m)}  \frac{\#\ol{\Sigma}_f^{\on{sol}}}{p^{\dim V}} = \sum_{m = 1}^{2n+1} \sum_{f \in \mc{I}_p(m)}  \sum_{T \in G(\BZ/p\BZ)\backslash \ol{\Sigma}_f^{\on{sol}}} \frac{\#\mc{O}(T)}{p^{2n^2+3n+1}}, \nonumber
\end{align}
where for each $T \in V(\BZ/p\BZ)$, we denote by $\mc{O}(T)$ the $G(\BZ/p\BZ)$-orbit of $T$. Let $\on{Stab}(T) \subset G(\BZ/p\BZ)$ denote the stabilizer of $T$. Substituting $\#\mc{O}(T) = \#G(\BZ/p\BZ)/\#\on{Stab}(T)$ into~\eqref{eq-porbit1} yields that
\begin{equation} \label{eq-porbit2}
\int_{T \in V(\BZ_p)} \frac{\psi_p(T)}{m_p(T)} dT  \leq \sum_{m = 1}^{2n+1} \sum_{f \in \mc{I}_p(m)}   \sum_{T \in G(\BZ/p\BZ)\backslash \ol{\Sigma}_f^{\on{sol}}} \frac{1}{\#\on{Stab}(T)}\cdot \frac{\#G(\BZ/p\BZ)}{p^{2n^2+3n+1}}.
\end{equation}
The following lemma gives a formula for $\#\on{Stab}(T)$ that is independent of $f \in \mathscr{I}_p(m)$:
\begin{lemma} \label{lem-stablowerpodd}
For $T \in \ol{\Sigma}_f^{\on{sol}}$, we have that $\#\on{Stab}(T) = 2^{m-1}$.
\end{lemma}
\begin{proof}[Proof of Lemma~\ref{lem-stablowerpodd}]
Let $K_f = (\BZ/p\BZ)[x]/(f)$. By Proposition~\ref{prop-slo}, $\#\on{Stab}(T) = \#K_f^\times[2]_{\on{N}\equiv 1}$. 
Letting $f$ split over $\BZ/p\BZ$ as $f(x) = \prod_{i = 1}^m f_i(x)^{n_i}$, where the $f_i$ are distinct and irreducible, we have that $K_f \simeq \prod_{i = 1}^m (\BZ/p\BZ)[x]/(f_i^{n_i})$. Consequently, $K_f^\times[2]$ contains the $2^m$ elements of the form $(\pm 1, \dots, \pm 1)$. Exactly half of these elements have norm $1$, so $\#K_f^\times[2]_{\on{N}\equiv 1} = 2^{m-1}$.
\end{proof}
To bound the right-hand side of~\eqref{eq-porbit2}, it remains to control the number of $G(\BZ/p\BZ)$-orbits on $\ol{\Sigma}_f^{\on{sol}}$. But there can only be as many orbits as there are points in $\BP^1(\BZ/p\BZ)$, which has size $p+1$. Combining this observation with~\eqref{eq-porbit2} and Lemma~\ref{lem-stablowerpodd} yields the proposition.
\end{proof}
\begin{remark}
It is well-known (see~\cite[\S6.1]{MR3156850}) that for each odd prime $p$, we have
\begin{equation} \label{eq-gsize}
\#G(\BZ/p\BZ) = p^{n^2} \cdot \prod_{i = 1}^n (p^{2i}-1) \leq p^{2n^2+n}.
\end{equation}
Further note that $\on{Vol}(G(\BZ_p)) = \#G(\BZ/p\BZ)/p^{2n^2+n}$.
\end{remark}

\subsection{Step (B): Bounding the Mass at $2$ when $2 \nmid f_0$} \label{sec-orbates}

Here, it turns out to be ineffective to work over $\BZ/2\BZ$ or even $\BZ/4\BZ$; instead, we work over $\BZ/8\BZ$. 
Let $\mc{I}_8(m)$ denote the set of monic degree-$(2n+1)$ polynomials in $(\BZ/8\BZ)[x]$ having $m$ distinct irreducible factors over $\BZ/2\BZ$. When $2 = p \nmid f_0$, we obtain the following result, which shows that the mass at $2$ in~\eqref{eq-deltamiddle} contributes a saving of $O(2^{-2n})$:
\begin{proposition}  \label{prop-2denstoobig}
We have that
\begin{equation*}
  \int_{T \in V(\BZ_2)} \frac{\psi_2(T)}{m_2(T)} dT  \leq \min\left\{1,\sum_{m = 1}^{2n+1}  \frac{12}{2^{2n+m-1}} \cdot 2 \cdot \frac{\#\mc{I}_8(m)}{8^{2n+1}}\right\}
\end{equation*}
\end{proposition}
\begin{proof}
Given a monic degree-$(2n+1)$ polynomial $f$ over $\BZ/8\BZ$, let $\ol{\Sigma}_f^{\on{sol}} \subset V(\BZ/8\BZ)$ denote the closure under the action of $G(\BZ/8\BZ)$ of the mod-$8$ reduction of the subset $\{T \in \scr{P}_2 : \on{ch}(T) \equiv f \pmod8\}$ (note that this subset may not be \emph{a priori} $G(\BZ/8\BZ)$-invariant, because $G$ is not smooth over $\BZ_2$ and the map $G(\BZ_2) \to G(\BZ/8\BZ)$ is not surjective). Then we have
\begin{align} \label{eq-2orbit1}
\int_{T \in V(\BZ_2)} \frac{\psi_2(T)}{m_2(T)} dT & \leq \int_{T \in \mathscr{P}_2} dT \\ & \leq \sum_{m = 1}^{2n+1} \sum_{f \in \mc{I}_8(m)}  \frac{\#\ol{\Sigma}_f^{\on{sol}}}{8^{\dim V}} = \sum_{m = 1}^{2n+1}\sum_{f \in \mc{I}_8(m)}  \sum_{T \in G(\BZ/8\BZ)\backslash \ol{\Sigma}_f^{\on{sol}}} \frac{\#\mc{O}(T)}{8^{2n^2+3n+1}}, \nonumber
\end{align}
where for each $T \in V(\BZ/8\BZ)$, we denote by $\mc{O}(T)$ the $G(\BZ/8\BZ)$-orbit of $T$. Let $\on{Stab}(T) \subset G(\BZ/8\BZ)$ denote the stabilizer of $T$. Substituting $\#\mc{O}(T) = \#G(\BZ/8\BZ)/\#\on{Stab}(T)$ into~\eqref{eq-2orbit1} yields that
\begin{equation} \label{eq-2orbit2}
\int_{T \in V(\BZ_2)} \frac{\psi_2(T)}{m_2(T)} dT \leq \sum_{m = 1}^{2n+1} \sum_{f \in \mc{I}_8(m)}   \sum_{T \in G(\BZ/8\BZ)\backslash \ol{\Sigma}_f^{\on{sol}}} \frac{1}{\#\on{Stab}(T)}\cdot \frac{\#G(\BZ/8\BZ)}{8^{2n^2+3n+1}}.
\end{equation}
The following lemma gives a lower bound for $\#\on{Stab}(T)$ that is independent of $f \in \mathscr{I}_8(m)$:
\begin{lemma}\label{lem-stabbound}
For $T \in \ol{\Sigma}_f^{\on{sol}}$, we have that $\#\on{Stab}(T) \geq 2^{2n+m-1}$.
\end{lemma}
\begin{proof}[Proof of Lemma~\ref{lem-stabbound}]
Let $R_f = (\BZ/8\BZ)[x]/(f)$. By Remark~\ref{rem-robust}, $\on{Stab}(T)$ contains a subgroup isomorphic to $R_f^\times[2]_{\on{N}\equiv 1} \defeq \{\rho \in R_f^\times : \rho^2 = 1 = \on{N}(\rho)\}$.
Let $f$ split over $\BZ/8\BZ$ as $f(x) = \prod_{i = 1}^m f_i(x)$, where the $f_i$ have the property that their mod-$2$ reductions are powers of distinct irreducible polynomials. Then it follows from~\cite[Theorem~7]{MR4101402} together with the Chinese Remainder Theorem that $R_f \simeq \prod_{i = 1}^m (\BZ/8\BZ)[x]/(f_i)$, and $R_f^\times[2]$ contains the $2^{2n+m+1}$ elements of the form
$$\left(\dots, \rho_0 + \sum_{j = 1}^{d_i-1} \rho_j\theta_i^j, \dots\right) \in \prod_{i = 1}^m (\BZ/8\BZ)[x]/(f_i),$$
where $\rho_0$ is any element of $(\BZ/8\BZ)^\times$ and $\rho_j$ is any element of $\{0,4\} \subset \BZ/8\BZ$ for each $j$, $\theta_i$ is the image of $x$ in $(\BZ/8\BZ)[x]/(f_i)$, and $d_i = \deg f_i$ for each $i$. At least $\frac{1}{4}$ of these elements have norm $1$, so $\#\on{Stab}(T) \geq \#R_F^\times[2]_{\on{N}\equiv 1} \geq 2^{2n+m-1}$.
\end{proof}

Let $\widehat{G}$ be the split odd orthogonal group scheme, defined in the same way as $G$ but without the determinant-$1$ condition. Observe that the determinant map $\on{det} \colon \widehat{G}(\BZ/8\BZ) \to (\BZ/8\BZ)^\times$ is surjective: for each $\sigma \in (\BZ/8\BZ)^\times$, the diagonal matrix with row-$(n+1)$, column-$(n+1)$ entry equal to $\sigma$ and all other diagonal entries equal to $1$ is orthogonal with respect to $A_0$ and has determinant $\sigma$. Combining this observation with the computation of $\#\widehat{G}(\BZ/8\BZ)$ in Proposition~\ref{prop-theprime2case} (to follow) yields
\begin{equation} \label{eq-easy8bound}
\frac{\#G(\BZ/8\BZ)}{8^{2n^2+n}} = \frac{\#\widehat{G}(\BZ/8\BZ)}{4 \cdot 8^{2n^2+n}} = 2 \cdot \frac{(2^n-1)\cdot\prod_{i = 1}^{n-1}(2^{2i}-1)}{2^{n^2}}\leq 2 \cdot 1 = 2.
\end{equation}
To bound the right-hand side of~\eqref{eq-2orbit2}, it remains to control the number of $G(\BZ/8\BZ)$-orbits on $\ol{\Sigma}_f^{\on{sol}}$. But there can only be as many orbits as there are points in $\mathbb{P}^1(\BZ/8\BZ)$, which has size $12$. Combining this observation with~\eqref{eq-2orbit2}, Lemma~\ref{lem-stabbound}, Proposition~\ref{prop-theprime2case}, and~\eqref{eq-easy8bound} yields the proposition.
\end{proof}

In the following proposition, we determine the value of $\#\widehat{G}(\BZ/8\BZ)$, which was used in~\eqref{eq-easy8bound}:
\begin{proposition} \label{prop-theprime2case}
We have that
$$\#\widehat{G}(\BZ/8\BZ) = 2^{5n^2+3n+3}\cdot(2^n-1)\cdot\prod_{i = 1}^{n-1}(2^{2i}-1)$$
when $n \geq 1$, and $\#\widehat{G}(\BZ/8\BZ) = 4$ when $n = 0$.
\end{proposition}
\begin{proof}
We apply the recursive formula for computing sizes of orthogonal groups modulo $8$ given in~\cite[\S4.1.3]{MR83960}. The first step is to diagonalize $A_0$. One readily verifies that there exists $g_1 \in \on{GL}_{2n+1}(\BZ)$ such that $g_1A_0g_1^T$ is equal to the diagonal matrix with first $n+1$ diagonal entries equal to $1$ and remaining $n$ diagonal entries equal to $-1$. By~\cite[proof of Lemma~3]{MR12640}, there exists $g_2 \in \on{GL}_{2n+1}(\BZ_2)$ such that $g_2g_1A_0g_1^Tg_2^T$ is equivalent modulo $8$ to the following diagonal matrix:
\begin{equation} \label{eq-diagform}
    \left[\begin{array}{cccccc} 1 & \cdots & 0 & 0 & 0 & 0 \\
    \vdots & \ddots & \vdots & \vdots & \vdots & \vdots \\
    0 & \cdots & 1 & 0 & 0 & 0 \\
    0 & \cdots & 0 & a_1 & 0 & 0 \\
    0 & \cdots & 0 & 0 & a_2 & 0 \\
    0 & \cdots & 0 & 0 & 0 & a_3
    \end{array}\right], \,\, \text{where} \,\, (a_1, a_2, a_3) = \begin{cases} (1,1,1) & \text{ if $2n+1 \equiv 1 \pmod 8$,} \\ (1,1,7) & \text{ if $2n+1 \equiv 3 \pmod 8$,} \\ (1,3,3) & \text{ if $2n+1 \equiv 5 \pmod 8$,} \\ (3,3,7) & \text{ if $2n+1 \equiv 7 \pmod 8$} \end{cases}
\end{equation}
Using~\eqref{eq-diagform} together with the recursive formula displayed in~\cite[\S4.1.3]{MR83960} yields that
\begin{equation} \label{eq-intermedform}
\#\widehat{G}(\BZ/8\BZ) = \prod_{j = 1}^{2n+1} \big(8^j+4^{j+1}\cdot f(u_{2n+1-j})+2(4\sqrt{2})^j \cdot \cos(\tfrac{\pi}{4} K_{2n+1-j})-8 \cdot 4^jh(u_{2n+1-j})\big),
\end{equation}
where the functions $f$ and $h$ are defined by
$$f(u_i) = \begin{cases} 1 & \text{ if $u_i \equiv 0 \pmod 8$,} \\ -1 & \text{ if $u_i \equiv 4 \pmod 8$,} \\ 0 & \text{ otherwise} \end{cases} \quad \text{and} \quad h(u_i) = \begin{cases} 1 & \text{ if $i < 2n$ and $u_i \equiv 0 \pmod 8$,} \\ 0 & \text{ otherwise} \end{cases}$$
and where the quantities $u_i$ and $h_i$ are given as follows:

\vspace*{0.2cm}
\noindent \emph{Case 1: $(a_1, a_2, a_3) = (1,1,1)$}. We have that
\begin{equation} \label{eq-uk1}
    u_{2n+1-j} = j-1 \quad \text{and} \quad K_{2n+1-j} = j-2.
\end{equation}

\vspace*{0.2cm}
\noindent \emph{Case 2: $(a_1, a_2, a_3) = (1,1,7)$}. We have that
\begin{equation} \label{eq-uk2}
    u_{2n+1-j} = \begin{cases} j+5 & \text{ if $j \geq 2$,} \\ 0 & \text{ if $j = 1$}  \end{cases}\quad \text{and} \quad K_{2n+1-j} = \begin{cases} j-4 & \text{ if $j \geq 2$,}\\ -15 & \text{ if $j = 1$} \end{cases}
\end{equation}

\vspace*{0.2cm}
\noindent \emph{Case 3: $(a_1, a_2, a_3) = (1,3,3)$}. We have that
\begin{equation} \label{eq-uk3}
    u_{2n+1-j} = \begin{cases} j+3 & \text{ if $j \geq 3$,} \\ 3 & \text{ if $j = 2$,} \\ 0 & \text{ if $j = 1$} \end{cases} \quad \text{and} \quad K_{2n+1-j} = \begin{cases} j-6 & \text{ if $j \geq 3$,} \\ -8 & \text{ if $j = 2$,} \\ -7 & \text{ if $j = 1$} \end{cases}
\end{equation}

\vspace*{0.2cm}
\noindent \emph{Case 4: $(a_1, a_2, a_3) = (3,3,7)$}. We have that
\begin{equation} \label{eq-uk4}
   u_{2n+1-j} = \begin{cases} j+9 & \text{ if $j \geq 4$,} \\ 10 & \text{ if $j = 3$,} \\ 7 & \text{ if $j = 2$,} \\ 0 & \text{ if $j = 1$} \end{cases} \quad \text{and} \quad K_{2n+1-j} = \begin{cases} j-8 & \text{ if $j \geq 4$,} \\ -9 & \text{ if $j = 3$,} \\ -8 & \text{ if $j = 2$,} \\ -15 & \text{ if $j = 1$} \end{cases}
\end{equation}

\noindent Substituting~\eqref{eq-uk1},~\eqref{eq-uk2},~\eqref{eq-uk3}, and~\eqref{eq-uk4} into~\eqref{eq-intermedform} and simplifying yields the desired formula.

We can also compute $\#\widehat{G}(\BZ/8\BZ)$ by comparing two different formulas for a quantity known as ``the $2$-adic density of the quadratic lattice'' defined by the matrix $A_0$, which is up to normalization the same as the $2$-adic volume of $G(\BZ_2)$ with respect to Haar measure. We denote this quantity by $\alpha_2$. One of the two formulas for $\alpha_2$ is given in~\cite{MR3347991}, where Cho constructs a smooth model of the group scheme $G$ over $\BZ_2$. Letting $\wt{G}$ denote the special fiber of this smooth model, one can show by applying~\cite[Lemma~4.2, Remark~4.3, and Theorem 5.2]{MR3347991} that
\begin{equation} \label{eq-cho}
\alpha_2 = 2 \cdot 2^{-2n^2-n} \cdot \#\wt{G}(\BZ/2\BZ) = 2^{-2n^2+n+3} \cdot \#\on{SO}_{2n}(\BZ/2\BZ),
\end{equation}
where $\on{SO}_{2n}$ denotes the special orthogonal group on a $2n$-dimensional split orthogonal space. To be precise, $\alpha_2$ is defined in~\cite{MR3347991} to be $1/2$ of the quantity in~\eqref{eq-cho}, to account for the number of connected components of $G$ in $\widehat{G}$. However, we have removed this factor of $1/2$ to make the definition of $\alpha_2$ consistent with the one given by Conway and Sloane in~\cite[\S12]{MR965484}, \mbox{who define $\alpha_2$ by}
\begin{equation} \label{eq-conway}
    \alpha_2 = \frac{\#\widehat{G}(\BZ/2^r\BZ)}{(2^r)^{2n^2 + n}}
\end{equation}
for any sufficiently large positive integer $r$. It follows from~\cite[Proposition~5.6.1(ii) and proof of Lemma~5.6.5]{MR1245266} that we can take $r = 3$ in~\eqref{eq-conway}. Comparing~\eqref{eq-cho} and~\eqref{eq-conway} together with the fact that $\#\on{SO}_{2n}(\BZ/2\BZ) = 2^{n^2-n}\cdot(2^n-1)\cdot\prod_{i = 1}^{n-1}(2^{2i}-1)$ for $n \geq 1$ (see~\cite[\S6]{MR3347991}) yields that
\begin{equation*}
\#\widehat{G}(\BZ/8\BZ) = 2^{4n^2+4n+3} \cdot \#\on{SO}_{2n}(\BZ/2\BZ) = 2^{5n^2+3n+3}\cdot(2^n-1)\cdot\prod_{i = 1}^{n-1}(2^{2i}-1).\qedhere
\end{equation*}
\end{proof}
\begin{remark}
The strategy for the second proof of Proposition~\ref{prop-theprime2case} can likewise be used to obtain a formula for $\#\widehat{G}(\BZ/2^r\BZ)$ for each $r \geq 3$.
\end{remark}


\subsection{Step (B): Bounding the Mass at $p$ when $p \mid f_0$} \label{sec-overzpqp}

When $p \nmid f_0$, Propositions~\ref{prop-denstoobig} and~\ref{prop-2denstoobig} give sufficiently good control on the size of $\mathscr{P}_p$. But when $p \mid f_0$, the bounds in Propositions~\ref{prop-denstoobig} and~\ref{prop-2denstoobig} do not suffice. Indeed, since proving Theorems~\ref{thm-main} and~\ref{thm-hasse} involves averaging over \emph{monicizations} of forms in $\mathscr{F}_{2n+1}(f_0)$, we must obtain a bound that decays with $|f_0|_p$ at least as fast as $|f_0|_p^{2n^2+n}$, which is the $p$-adic density of the set of monicizations of forms in $\mathscr{F}_{2n+1}(f_0,\BZ_p)$.

In this section, we bound the mass at $p$ for primes $p \mid f_0$ by controlling the number of orbits of $G(\BQ_p)$ on $V(\BQ_p)$ that can arise from points of $\mathscr{S}_F(\BZ_p)$. Specifically, we obtain the following result, which shows that the mass at $p$ is bounded by $O(|f_0|_p^{2n^2+n})$ when $p \neq 2$ and by $O(2^{-n} \cdot |f_0|_2^{2n^2+n})$ when $p = 2$:
\begin{proposition} \label{prop-qpbound}
    For each prime $p \mid f_0$, we have that
\begin{align*}
    \int_{T \in V(\BZ_p)} \frac{\psi_{p}(T)}{m_p(T)} dT & \leq \begin{cases} 4 \cdot \on{Vol}(G(\BZ_p)) \cdot |f_0|_p^{2n^2+n}, & \text{if $p \neq 2$,} \\
8 \cdot 2^n \cdot |\mc{J}|_2 \cdot \on{Vol}(G(\BZ_2)) \cdot |f_0|_2^{2n^2+n}, & \text{if $p = 2$.} 
    \end{cases}
\end{align*}
\end{proposition}
\begin{proof}
We start with the following lemma, which gives a convenient formula for the mass at $p$:
\begin{lemma} \label{prop-intoverp}
Let $\varphi$ be a continuous $G(\BQ_p)$-invariant function on $V(\BZ_p)$ such that every element $T \in V(\BZ_p)$ in the support of $\varphi$ is separable. Then we have
\begin{align*}
& \int_{T \in V(\BZ_p)} \frac{\varphi(T)}{m_p(T)} dT = \\
& \quad |\mc{J}|_p \cdot \on{Vol}(G(\BZ_p)) \cdot \int_{\substack{(c_1, \dots, c_{2n+1}) \in \BZ_p^{2n+1} \\ }} \sum_{\substack{T \in G(\BQ_p)\backslash V(\BZ_p) \\ \on{ch}(T) = x^{2n+1} + c_1x^{2n}+\cdots+c_{2n+1}}}\frac{\varphi(T)}{\#\on{Stab}_{\BQ_p}(T)}\,\,dc_1\cdots dc_{2n+1},
\end{align*}
where $\mc{J} \in \BQ$ is the same constant that appears in Theorem~\ref{thm-thecount}.
\end{lemma}
\begin{proof}[Proof of Lemma~\ref{prop-intoverp}]
The proof is identical to that of~\cite[Corollary~11.3]{MR3156850}.
\end{proof}
To evaluate the sum within the integral that occurs in the formula given by Lemma~\ref{prop-intoverp}, we utilize the following lemma: 
\begin{lemma} \label{prop-qporbit}
Let $p$ be an odd $($resp., even$)$ prime, and suppose that $F \in \scr{F}_{2n+1}(f_0,\BZ_p)$ splits as a product $m$ distinct irreducible factors over $\BZ_p$. Then there are at most $2^{m+1}$ $($resp., $2^{n+m+2})$ orbits of $G(\BQ_p)$ on $V(\BQ_p)$ arising from points of $\mathscr{S}_F(\BZ_p)$ via the construction in \S\ref{sec-buildabear}. The size of the stabilizer in $G(\BQ_p)$ of any such orbit is at least $\#K_F^{\times}[2]_{\on{N}\equiv 1} = 2^{m-1}$.
\end{lemma}
\begin{proof}[Proof of Lemma~\ref{prop-qporbit}]
By Proposition~\ref{prop-rationalcorresp}, it suffices to bound the number of elements of $(K_F^\times/K_F^{\times 2})_{\on{N} \equiv 1}$ of the form $f_0 \cdot \delta$, where $\delta \in K_F^\times$ is the second coordinate of a pair $(I,\delta) \in H_F$ arising from a points of $\scr{S}_F(\BZ_p)$ via the construction in \S\ref{sec-buildabear}.

Fix $z_0 \in \BZ_p$, and consider the monic odd-degree hyperelliptic curve $C_{F,z_0}$ of genus $n$ over $\BQ_p$ defined by the affine equation $C_{F,z_0} \colon y^2 = F_{\on{mon}}(x, z_0)$. Notice that $C_{F,z_0}$ is a quadratic twist of $C_{F,1}$ and that for $z_0' \in \BZ_p$, the curves $C_{F,z_0}$ and $C_{F,z_0'}$ are isomorphic over $\BQ_p$ when $z_0$ and $z_0'$ represent the same class in $\BQ_p^\times/\BQ_p^{\times 2}$. Let $\{\eta_1, \dots, \eta_\ell\} \subset \BZ_p$ denote a complete set of representatives of elements of $\BQ_p^\times/\BQ_p^{\times 2}$, and note that we can take $\ell = 4$ when $p$ is odd and $\ell = 8$ when $p = 2$. Then each $C_{F,z_0}$ is isomorphic to one of $C_{F,\eta_1}, \dots, C_{F, \eta_\ell}$.

Let $(x_0, y_0, z_0) \in \scr{S}_F(\BZ_p)$ giving rise to a pair $(I,\delta)$. Then the point $(f_0x_0, y_0)$ is a $\BQ_p$-rational point on the curve $C_{F,z_0}$ and hence is identified with a $\BQ_p$-rational point $(x_0',y_0')$ on one of the curves $C_{F,\eta_i}$. In~\cite[\S5]{MR3156850}, Bhargava and Gross show that points in $C_{F,\eta_i}(\BQ_p)$ naturally give rise to orbits of $G(\BQ_p)$ on $V(\BQ_p)$ having characteristic polynomial $F_{\on{mon}}(x, \eta_i)$.  Under the construction of Bhargava and Gross, points of $C_{F,\eta_i}(\BQ_p)$ give rise to $G(\BQ_p)$-orbits on $V(\BQ_p)$ via the following composition:
\begin{equation} \label{eq-cassels}
C_{F,\eta_i}(\BQ_p) \to J_i(\BQ_p)/2J_i(\BQ_p) \to H^1(G_{\BQ_p},J_i[2]) \simeq (K_F^\times/K_F^{\times 2})_{\on{N}\equiv1},
\end{equation}
where $J_i$ denotes the Jacobian of $C_{F,\eta_i}$ and where we identify the set of $G(\BQ_p)$-orbits on $V(\BQ_p)$ having characteristic polynomial $F_{\on{mon}}(x,z_0)$ with $(K_F^\times/K_F^{\times 2})_{\on{N}\equiv1}$ via Proposition~\ref{prop-rationalcorresp}. The composite map in~\eqref{eq-cassels} is known as the \emph{Cassels map}, and was first studied by Cassels in~\cite{MR717589}.

Let $\theta' \in K_F$ be the image of $x$ under the identification $\BQ_p[x]/F_{\on{mon}}(x, 1) \simeq K_F$. If $y_0 = 0$, then $F_{\on{mon}}(x,\eta_i\cdot z)$ factors uniquely as
$$F_{\on{mon}}(x,\eta_i\cdot z) = \big(x - f_0\tfrac{x_0}{z_0}\eta_i \cdot z\big) \cdot \wt{F}_{\on{mon}}(x,\eta_i \cdot z)$$
where $\wt{F}_{\on{mon}} \in \mathscr{F}_{2n}(1, \BZ_p)$. (The notation $\wt{F}_{\on{mon}}$ makes sense, because $\wt{F}_{\on{mon}}$ is indeed the monicized form of $\wt{F}$ as defined in~\eqref{eq-defftilde}.) Using the explicit description of the Cassels map stated in~\cite[\S2]{MR2521292}, one finds that the point $(x_0', c') \in C_{F,\eta_i}(\BQ_p)$ gives rise to the class
$$\begin{rcases*} x_0' - \theta' \eta_i & \text{ if $y_0 \neq 0$,} \\ \wt{F}_{\on{mon}}(\theta' \eta_i, \eta_i) + (x_0' - \theta' \eta_i)  & \text{ if $y_0 = 0$}\end{rcases*} \equiv f_0 \cdot \delta \in (K_F^\times/K_F^{\times 2})_{\on{N}\equiv1}.$$
 Meanwhile, by Proposition~\ref{prop-rationalcorresp}, the class in $(K_F^\times/K_F^{\times 2})_{\on{N}\equiv1}$ corresponding to $(x_0,y_0, z_0) \in \mathscr{S}_F(\BZ_p)$ is also $f_0 \cdot \delta$. Thus, the number of elements of $(K_F^\times/K_F^{\times 2})_{\on{N} \equiv 1}$ of the form $f_0 \cdot \delta$ arising form points of $\mathscr{S}_F(\BZ_p)$ is bounded above by the total number of $G(\BQ_p)$-orbits on $V(\BQ_p)$ arising from $\BQ_p$-rational points on any one of the finitely many curves $C_{F,\eta_1}, \dots, C_{F, \eta_\ell}$. By~\cite[\S6.2]{MR3156850}, the number of $G(\BQ_p)$-orbits on $V(\BQ_p)$ arising from each $C_{F,\eta_i}$ is at most $2^{m_i-1}$ when $p$ is odd (resp., $2^{n+m_i-1}$ when $p = 2$), where $m_i$ is the number of distinct irreducible factors of $F_{\on{mon}}(x,\eta_i)$. But clearly $m_i = m$ for each $i$, so we get the desired upper bound of $\ell \cdot 2^{m-1}$ when $p$ is odd (resp., $2^{n+m-1}$ when $p = 2$).

As for the statement about stabilizers, by Theorem~\ref{thm-theconstruction}, the stabilizer of any orbit arising via $\mathsf{orb}_F$ contains a subgroup isomorphic to $K_F^\times[2]_{\on{N}\equiv1}$, which has size $\#K_F^\times[2]_{\on{N}\equiv1} = 2^{m-1}$.
\end{proof}

If $p \neq 2$, combining Lemmas~\ref{prop-intoverp} and~\ref{prop-qporbit} along with~\eqref{eq-podd} yields that
\begin{align*}
\int_{T \in V(\BZ_p)} \frac{\psi_{p}(T)}{m_p(T)} dT & \leq |\mc{J}|_p \cdot \on{Vol}(G(\BZ_p)) \cdot \int_{\substack{\vec{c} = (c_1, \dots, c_{2n+1}) \in \BZ_p^{2n+1} \\ f_0^{i-1} \mid c_i \text{ for every }i }} \frac{2^{m(\vec{c})+1}}{2^{m(\vec{c})-1}}\,\,dc_1\cdots dc_{2n+1} \\
& = 4 \cdot \on{Vol}(G(\BZ_p)) \cdot |f_0|_p^{2n^2+n},
\end{align*}
    as necessary, where $m(\vec{c})$ is the number of irreducible factors of the polynomial $x^{2n+1}+\sum_{i = 1}^{2n+1} c_i \cdot x^{2n+1-i} \in \BZ_p[x]$. The case where $p = 2$ may be deduced analogously.

 This completes the proof of Proposition~\ref{prop-qpbound}.
\end{proof}

Now, let $p$ be an odd prime. If we restrict our consideration to those forms $F \in \mc{F}_{2n+1}(f_0, \BZ_p)$ such that $R_F$ is maximal, we can do better than the bound in Proposition~\ref{prop-qpbound}. Indeed, let $\mu_p^{\max}(f_0)$ denote the $p$-adic density of $F \in \mc{F}_{2n+1}(f_0, \BZ_p)$ such that $R_F$ is maximal, and let $\psi_p^{\max} \colon V(\BZ_p) \to \{0,1\} \subset \BR$ denote the indicator function of elements $T \in \mathscr{P}_p$ that arise from points $\mathscr{S}_F(\BZ_p)$ for a form $F \in \mc{F}_{2n+1}(f_0, \BZ_p)$ such that $R_F$ is maximal. Then we have the following result:
\begin{proposition} \label{prop-zpbound}
    For each prime $p \mid f_0$, we have that
\begin{align*}
    \int_{T \in V(\BZ_p)} \frac{\psi_p^{\max}(T)}{m_p(T)} dT & \leq \on{Vol}(G(\BZ_p)) \cdot |f_0|_p^{2n^2+n} \cdot\mu_p^{\max}(f_0).
\end{align*}
\end{proposition}
\begin{proof}
It follows from Lemma~\ref{prop-intoverp} that
\begin{align} \label{eq-zpmassbound}
    & \int_{T \in V(\BZ_p)} \frac{\psi_{p}^{\max}(T)}{m_p(T)} dT \leq \\
    & \quad |\mc{J}|_p \cdot \on{Vol}(G(\BZ_p)) \cdot \int_{\substack{(c_1, \dots, c_{2n+1}) \in \BZ_p^{2n+1} \\ }} \sum_{\substack{T \in G(\BZ_p)\backslash V(\BZ_p) \\ \on{ch}(T) = x^{2n+1} + c_1x^{2n}+\cdots+c_{2n+1}}}\frac{\psi_p^{\max}(T)}{\#\on{Stab}_{\BZ_p}(T)}\,\,dc_1\cdots dc_{2n+1}, \nonumber
\end{align}
When $F$ is such that $R_F$ is the maximal order in $K_F$, we have the following analogue of Lemma~\ref{prop-qporbit} bounding the number of orbits over $\BZ_p$ arising from points in $\mathscr{S}_F(\BZ_p)$:
\begin{lemma} \label{prop-overzp}
Let $p$ be an odd prime, and suppose that $R_F$ is maximal in $K_F$. Then there are at most $2^{m-1}$ orbits of $G(\BZ_p)$ on $V(\BZ_p)$  arising from points of $\mathscr{S}_F(\BZ_p)$ via the construction in \S\ref{sec-buildabear}. The size of the stabilizer in $G(\BZ_p)$ of any such orbit is $2^{m-1}$.
\end{lemma}
\begin{proof}[Proof of Lemma~\ref{prop-overzp}]
By~\cite[Theorem~30]{Swpreprint}, the number of $G(\BZ_p)$-orbits on $V(\BZ_p)$ arising from elements of $H_F$ via the correspondence in Theorem~\ref{thm-theconstruction} is equal to $\#(R_F ^\times/R_F^{\times 2})_{\on{N}\equiv1}=2^{m-1}$. By Theorem~\ref{thm-theconstruction}, the stabilizer of any such orbit has size $\#R_F^\times[2]_{\on{N}\equiv1} = 2^{m-1}$.
\end{proof}
Proposition~\ref{prop-zpbound} now follows by substituting the result of Lemma~\ref{prop-overzp} into the bound in~\eqref{eq-zpmassbound}.
\end{proof}

\subsection{Step (C): Finishing the Proof} \label{sec-hillary}

We can now apply Theorem~\ref{thm-cong2} and Propositions~\ref{prop-thisisforr},~\ref{prop-denstoobig},~\ref{prop-2denstoobig}, and~\ref{prop-qpbound} to obtain a bound on~\eqref{eq-deltamiddle}, and hence on~\eqref{eq-thekeybound}; combining the result with~\eqref{eq-numer1est} and~\eqref{eq-denomest} gives a bound on $\delta$ via~\eqref{eq-thekeybound0}. Then, simplifying the result, we draw the following conclusions:

\vspace*{0.2in}
\noindent\emph{Case 1: $2 \nmid f_0$}. When $2 \nmid f_0$, we have:
\begin{align}
& \delta - \mu_{f_0} \ll
 2^n \cdot 2^{-2n} \cdot  \prod_{p \nmid 2f_0}\min\left\{\frac{p^{2n^2+n}}{\#G(\BZ/p\BZ)},\sum_{m = 1}^{2n+1}  \frac{p+1}{2^{m-1}} \cdot \frac{\#\mc{I}_p(m)}{p^{2n+1}}\right\} \label{eq-46when2good} 
\end{align}
where 
the mass at $\infty$ contributes $O(2^n)$, the mass at $2$ contributes $O(2^{-2n})$, and the masses at primes dividing $f_0$ contribute $O(1)$.

We now bound the product of the masses at primes not dividing $2f_0$, proving that they contribute $O(2^{-\varepsilon_1 n^{\varepsilon_2}})$. The following lemma allows us to bound the factors at primes less than a small power of the degree $2n+1$:
\begin{lemma} \label{lem-poddbounds}
We have for each fixed $A \in (0,1/3) \subset \BR$ that
$$\prod_{ 3 \leq p \leq (2n+1)^A}\sum_{m = 1}^{2n+1}  \frac{p+1}{2^{m-1}} \cdot \frac{\#\mc{I}_p(m)}{p^{2n+1}} \ll 2^{-\varepsilon_1 n^{\varepsilon_2}},$$
 where $\varepsilon_1, \varepsilon_2 > 0$ are real numbers that may depend on $A$.
\end{lemma}
\begin{proof}
We first estimate each factor in the product over primes. We split the sum at the prime $p$ into two ranges, one for $m \leq \frac{1}{2}\log(2n+1)$ and one for $m > \frac{1}{2}\log(2n+1)$. We bound the sum over $m > \frac{1}{2}\log(2n+1)$ as follows:
\begin{align}
   \sum_{\frac{1}{2}\log(2n+1) < m \leq 2n+1}  \frac{p+1}{2^{m-1}} \cdot \frac{\#\mc{I}_p(m)}{p^{2n+1}} & \leq \sum_{\frac{1}{2}\log(2n+1) < m \leq 2n+1}  \frac{p+1}{2^{m-1}}  \leq \frac{2(p+1)}{\sqrt{2n+1}}.\label{eq-splitsum1}
\end{align}
For the sum over $m < \frac{1}{2}\log(2n+1)$, we rely on the following result of Afshar and Porritt:\footnote{Note that a result of this type was first proven by Car in~\cite{MR651431}; the dependence of the error term on the prime $p$ was made explicit in~\cite{MR3897499}.}
\begin{theorem}[\protect{\cite[Remark~2.11]{MR3897499}}] \label{thm-afsharporritt}
 Let $p$ be a prime, and let $1 \leq m \leq \log (2n+1)$. Then we have the following uniform estimate:
 $$\frac{\#\mc{I}_p(m)}{p^{2n+1}} = \frac{1}{(2n+1)} \cdot \frac{(\log (2n+1))^{m-1}}{(m-1)!} \cdot \left(D_p\left(\frac{m-1}{\log(2n+1)}\right)+O\left(\frac{m}{(\log (2n+1))^2}\right)\right)$$
 where the implied constant is absolute $($i.e., does not depend on $p${}$)$, and where the function $D_p$ is defined as follows. Letting $\mc{I} \subset (\mathbb{Z}/p\mathbb{Z})[x]$ be the set of all monic irreducible polynomials, we have
 $$D_p(z) = \frac{E(1/p,z)}{\Gamma(1+z)}, \quad \text{where} \quad E(x,z) = \prod_{f \in \mc{I}} \left(1 + \frac{zx^{\deg f}}{1-x^{\deg f}}\right)\cdot (1 - x^{\deg f})^z.$$
\end{theorem}
We now turn the very careful estimate in Theorem~\ref{thm-afsharporritt} into a more easily usable form. First, notice that for any $x \in (0,1/2) \subset \BR$ and $z \in (0,1)$ we have
\begin{equation} \label{eq-mathest1}
\left(1 + \frac{zx}{1-x}\right)\cdot (1 - x)^z \leq \left(1 + \frac{zx}{1-x}\right)\cdot (1 - zx) = 1 + \frac{(z-z^2)x^2}{1-x} \leq 1 + x^2.
\end{equation}
By Carlitz's Theorem (see~\cite{MR1506871}), there are exactly $p^{d-1}$ monic separable polynomials of degree $d$ over $\BZ/p\BZ$, and hence at most $p^{d-1}$ monic irreducible polynomials of degree $d$ over $\BZ/p\BZ$. Using this together with~\eqref{eq-mathest1}, we deduce that
\begin{equation} \label{eq-mathest2}
E(1/p,z) \leq \prod_{f \in \mc{I}} (1 + p^{-2 \deg f}) \leq \prod_{d = 1}^\infty (1 + p^{-2 d})^{p^{d-1}}.
\end{equation}
To estimate the right-hand side of~\eqref{eq-mathest2}, we take its logarithm and apply the bound $\log(1+x) \leq x$ (which holds for $x > -1$):
\begin{equation} \label{eq-mathest3}
\log\left(\prod_{d = 1}^\infty (1+p^{-2d})^{p^{d-1}}\right) = \sum_{d = 1}^\infty p^{d-1} \cdot \log(1+p^{-2d}) \leq \sum_{d = 1}^\infty p^{d-1} \cdot p^{-2d} = \frac{1}{p(p-1)}.
\end{equation}
Taking $p = 2$ in~\eqref{eq-mathest3} and using the fact that $\Gamma(1+z) > 1/2$, we have for any prime $p$, any $1 \leq m \leq \log(2n+1)$, and all sufficiently large $n$ that
\begin{equation} \label{eq-mathest4}
D_p\left(\frac{m-1}{\log(2n+1)}\right) + O\left(\frac{m}{(\log (2n+1))^2}\right) \leq 2\sqrt{e} +O\left(\frac{m}{(\log (2n+1))^2}\right) \leq 4,
\end{equation}
where $e$ denotes the usual base of the natural logarithm. Upon combining~\eqref{eq-splitsum1},~\eqref{eq-mathest4}, and Theorem~\ref{thm-afsharporritt}, we find for all sufficiently large $n$ that
\begin{align}
\sum_{m = 1}^{2n+1} \frac{p+1}{2^{m-1}} \cdot \frac{\#\mc{I}_p(m)}{p^{2n+1}} & \leq \frac{2(p+1)}{\sqrt{2n+1}} + \sum_{1 \leq m  \leq \frac{1}{2}\log(2n+1)} \frac{4(p+1)}{(2n+1)} \cdot \frac{(\frac{1}{2}\log (2n+1))^{m-1}}{(m-1)!} \label{eq-mathest5} \\
& \leq \frac{2(p+1)}{\sqrt{2n+1}} +\frac{4(p+1)}{2n+1} \cdot e^{\frac{1}{2}\log(2n+1)} = \frac{6p+6}{\sqrt{2n+1}} \leq \frac{8p}{\sqrt{2n+1}}, \nonumber
\end{align}
where we have used the fact that $\sum_{i = 0}^N \frac{x^i}{i!} \leq e^x$ for any $x > 0$ and integer $N \geq 0$. By Erd\H{o}s' proof of Bertrand's Postulate, we have $\prod_{p < N} < 4^N$ for any $N \geq 3$. By the Prime Number Theorem, for any $\varepsilon > 0$, there exists an integer $N > 0$ such that for all $N' > N$, the number of primes less than or equal to $N'$ is at least $(1-\varepsilon) \cdot \frac{N'}{\log N'}$ and is at most $(1+\varepsilon) \cdot \frac{N'}{\log N'}$. Combining this fact with~\eqref{eq-mathest5}, we find for any fixed $A < 1/3$ and all sufficiently large $n$ (where ``large'' depends on $A$ and $\varepsilon$) that
$$\prod_{ 3 \leq p \leq (2n+1)^A}\sum_{m = 1}^{2n+1}  \frac{p+1}{2^{m-1}} \cdot \frac{\#\mc{I}_p(m)}{p^{2n+1}} \leq \frac{8^{(1+\varepsilon) \cdot \frac{(2n+1)^A}{A \log (2n+1)}} \cdot 4^{(2n+1)^A}}{(\sqrt{2n+1})^{(1-\varepsilon) \cdot \frac{(2n+1)^A}{A  \log(2n+1)}}} \ll 2^{-\varepsilon_1 n^{\varepsilon_2}},$$
for some real numbers $\varepsilon_2, \varepsilon_2 > 0$.
\end{proof}
The following lemma allows us to bound the factors at the remaining primes:
\begin{lemma} \label{lem-volbound}
We have that $$\prod_{p > 2} \frac{p^{2n^2+n}}{\#G(\BZ/p\BZ)} \ll 1,$$ where the implied constant does not depend on $n$.
\end{lemma}
\begin{proof}
From~\eqref{eq-gsize}, it follows that
\begin{equation} \label{eq-oddvols1}
\log \frac{p^{2n^2+n}}{\#G(\BZ/p\BZ)} = \sum_{i = 1}^n -\log(1-p^{-2i}) \leq  2 \cdot \sum_{i = 1}^n p^{-2i} \leq  \frac{2}{p^2-1},
\end{equation}
where we have applied the bound $-\log(1-x) \leq 2x$, which holds for $x \in (0,1/2) \subset \BR$. It is not hard to check (by comparing derivatives) that $\frac{2}{p^2-1} \leq \log(1 + p^{-\frac{3}{2}})$ for each $p \geq 5$. Thus,
\begin{align*}\prod_{p > 2}\frac{p^{2n^2+n}}{\#G(\BZ/p\BZ)}  & \leq \frac{e^{\frac{1}{4}}}{(1+2^{-\frac{3}{2}})(1 + 3^{-\frac{3}{2}})} \cdot \prod_{p} (1 + p^{-\frac{3}{2}}) \ll 1. \qedhere\end{align*}
\end{proof}

It follows from Lemmas~\ref{lem-poddbounds} and~\ref{lem-volbound} that the factor at the odd primes not dividing $f_0$ is $O(2^{-\varepsilon_1n^{\varepsilon_2}})$ for some real numbers $\varepsilon_1, \varepsilon_2 > 0$, so by~\eqref{eq-46when2good}, we have $\delta = \mu_{f_0} + o(2^{-n})$. This completes the proof of the second statement in part (a) of Theorem~\ref{thm-main}.

\medskip

 We now determine the smallest $n$ for which our method yields that a positive proportion of superelliptic stacky curves in $\mathscr{F}_{2n+1}(f_0)$ are insoluble. To this end, let $\mathscr{F}_{2n+1}^{\max}(f_0)$ be the set of forms $F \in \mathscr{F}_{2n+1}(f_0)$ such that $R_F \otimes_\BZ \BZ_p$ is the maximal order in $K_F \otimes_\BQ \BQ_p$ for every prime $p \mid f_0$ and such that condition (b) in Theorem~\ref{thm-caseprime} fails for some prime $p \mid \upkappa$. Let $\delta_{\max}$ be the upper density of forms $F \in \scr{F}_{2n+1}^{\max}(f_0)$, enumerated by height, such that $\scr{S}_F(\BZ) \neq \varnothing$. For each prime $p \mid f_0$, instead of applying Proposition~\ref{prop-qpbound}, we can apply Proposition~\ref{prop-zpbound}, which yields the following explicit bound on $\delta_{\max}$:
\begin{align*} 
\delta_{\max} & \leq  3 \cdot 2^n \cdot \left(\sum_{m = 1}^{2n+1}  \frac{12}{2^{2n+m-1}} \cdot \frac{\#G(\BZ/8\BZ)}{8^{2n^2+n}} \cdot \frac{\#\scr{I}_8(m)}{8^{2n+1}} \right) \cdot \prod_{p > 2}\frac{p^{2n^2+n}}{\#G(\BZ/p\BZ)}
\end{align*}
Upon applying the estimate
$\#\scr{I}_8(m)/8^{2n+1} \leq 1$, we deduce that $\delta_{\max} < 2^{7-n} \leq 1$ whenever $n \geq 7$. By explicitly computing $\frac{\#\scr{I}_8(m)}{8^{2n+1}}$ in {\tt sage} for $n < 7$, one can check that we also have $\delta_{\max} < 1$ when $n \in \{5, 6\}$. Since the density of $\mathscr{F}_{2n+1}^{\max}(f_0)$ in $\mathscr{F}_{2n+1}(f_0)$ is positive, this completes the proof of the first statement in part (a) of Theorem~\ref{thm-main}.

\vspace*{0.1in}
\noindent\emph{Case 2: $2 \mid f_0$}. When $2 \mid f_0$, we bound the mass at $\infty$ slightly differently --- specifically, we do not use the bound on $|\mc{J}| \cdot \on{Vol}(G(\BZ) \backslash G(\BR))$ given by Lemma~\ref{lem-compJ}. Doing so, and applying the identities~\eqref{eq-normmultident} and~\eqref{eq-tam}, we obtain:
\begin{align}
& \delta - \mu_{f_0} \ll 2^{-n} \cdot 2^n \cdot \prod_{\substack{p \nmid f_0}}\min\left\{\frac{p^{2n^2+n}}{\#G(\BZ/p\BZ)} ,\sum_{m = 1}^{2n+1}  \frac{p+1}{2^{m-1}} \cdot \frac{\#\mc{I}_p(m)}{p^{2n+1}}\right\} \label{eq-46when2bad}
\end{align}
where the implied constant is independent of $n$. In~\eqref{eq-46when2bad}, the mass at $\infty$ contributes $O(2^{-n})$ (as the Jacobian factors at all places have canceled out), the mass at $2$ contributes $O(2^{n})$, and the masses at primes dividing $f_0$ contribute $O(1)$. It follows from Lemmas~\ref{lem-poddbounds} and~\ref{lem-volbound} that the factor at the primes not dividing $f_0$ is $O(2^{-\varepsilon_1n^{\varepsilon_2}})$ for some real numbers $\varepsilon_1, \varepsilon_2 > 0$, so $\delta = \mu_{f_0} + O(2^{-\varepsilon_1n^{\varepsilon_2}})$ in this case. This completes the proof of part (b) of Theorem~\ref{thm-main}.


\section{Obstruction to the Hasse Principle and the Proof of Theorem~\ref{thm-hasse}} \label{sec-bm}

In this section, we study how often superelliptic stacky curves satisfy the Hasse principle using the method of descent. Specifically, we describe $2$-coverings of superelliptic stacky curves, and we modify the proof of Theorem~\ref{thm-main} to show that superelliptic stacky curves often have no locally soluble $2$-coverings and thus fail the Hasse principle. Upon showing that such a $2$-covering obstruction to solubility is a special case of the Brauer--Manin obstruction, we obtain Theorem~\ref{thm-hasse}.

\subsection{Defining the Brauer--Manin Obstruction for $\mathscr{S}_F$} \label{sec-defbm}

Let $F$ be a separable integral binary form of degree $2n+1 \geq 3$. While a suitable theory of Brauer--Manin obstruction for stacks remains to be developed, such a theory is not required in the context of superelliptic stacky curves. Indeed, we can exploit the fact that the stack $\mathscr{S}_F$ is defined as a quotient of the punctured affine surface $\wt{S}_F$ (see~\eqref{eq-surfacedef}), where the usual theory of the Brauer--Manin obstruction for varieties applies.

Let $\wt{S}_{F,\BQ} = \wt{S}_F \otimes_{\BZ} \BQ$, and let $\wt{S}_{F,\BQ}(\BA_\BQ)^{\on{Br}}$ denote the usual Brauer--Manin set of the $\BQ$-variety $\wt{S}_{F,\BQ}$ (see~\cite[\S1]{BMadded} for the definition). Let
\begin{equation} \label{eq-needlabel1}
\wt{S}_F(\BA_\BQ)^{\on{int}} \defeq \prod_{v} \wt{S}_F(\BZ_v)
\end{equation}
be the set of adelic points of $\wt{S}_F$ that are integral in every place (i.e., the set of everywhere-primitive adelic solutions to the superelliptic equation $y^2 = F(x,z)$, or equivalently, the set of everywhere-integral adelic points on the stacky curve $\mathscr{S}_F$). With this notation in place, we define the \emph{Brauer--Manin set} of $\mathscr{S}_F$ by
\begin{equation} \label{eq-brdef}
\mathscr{S}_F(\BA_\BQ)^{\on{Br}} \defeq \wt{S}_F(\BA_\BQ)^{\on{int}} \cap \wt{S}_{F,\BQ}(\BA_\BQ)^{\on{Br}},
\end{equation}
and we say that $\mathscr{S}_F$ \emph{has a Brauer--Manin obstruction to having an integral point} if its Brauer--Manin set is empty (i.e.,~if $\mathscr{S}_F(\BA_\BQ)^{\on{Br}} = \varnothing$).

Because we have defined $\mathscr{S}_F(\BA_\BQ)^{\on{Br}}$ as a subset of $\wt{S}_{F,\BQ}(\BA_\BQ)^{\on{Br}}$, we can use the method of descent as laid out for open varieties in~\cite{BMadded} to obtain a set-theoretic ``upper bound'' on $\mathscr{S}_F(\BA_\BQ)^{\on{Br}}$. Indeed, suppose that we have a torsor $f \colon Z \to \wt{S}_{F,\BQ}$ by a group of multiplicative type or connected algebraic group $\Gamma$ (all defined over $\BQ$). Let
\begin{equation} \label{eq-needlabel2}
\wt{S}_{F,\BQ}(\BA_\BQ)^{f}\defeq \bigcup_{\sigma \in H^1(G_\BQ, \Gamma(\ol{\BQ}))} f^\sigma(Z^\sigma(\BA_\BQ))
\end{equation}
be the descent set of $f$, where $f^\sigma \colon Z^\sigma \to \wt{S}_{F,\BQ}$ is the twist of $f \colon Z \to \wt{S}_{F,\BQ}$ by a $1$-cocycle representing $\sigma \in H^1(G_\BQ, \Gamma(\ol{\BQ}))$. Since $\wt{S}_{F,\BQ}$ is smooth and geometrically integral,~\cite[Theorem~1.1]{BMadded} implies that
\begin{equation} \label{eq-brcont}
\wt{S}_{F,\BQ}(\BA_\BQ)^{\on{Br}} \subset \wt{S}_{F,\BQ}(\BA_\BQ)^f.
\end{equation}
It follows from~\eqref{eq-brdef} and~\eqref{eq-brcont} that $\mathscr{S}_F$ has a Brauer--Manin obstruction to having an integral point if it has an obstruction to descent by $\Gamma$-torsors in the sense that $\wt{S}_F(\BA_\BQ)^{\on{int}} \cap \wt{S}_{F,\BQ}(\BA_\BQ)^f = \varnothing$.

\subsection{$2$-Coverings} \label{sec-2covering}

Let $k$ be a field of characteristic not equal to $2$, let $f_0 \in k^\times$, and let $F \in \mathscr{F}_{2n+1}(f_0,k)$ be separable. Let $\theta_1, \dots, \theta_{2n+1} \in k_{\on{sep}}$ be the distinct roots of the polynomial $F(x,1)$, and for each $\sigma \in G_k$ and $i \in \{1, \dots, 2n+1\}$, let $\sigma(i) \in \{1, \dots, 2n+1\}$ be defined by $\theta_{\sigma(i)} = \sigma(\theta_i)$. Let $\wt{S}_F$ and $\mathscr{S}_F$ respectively denote the corresponding punctured affine surface and stacky curve (which are now defined over $k$, as opposed to $\BZ$). In this section, we specialize the discussion of descent in \S\ref{sec-defbm} to the case where $\Gamma$ is equal to the Jacobian $\on{Pic}^0(\mathscr{S}_F)$ of $\mathscr{S}_F$. Notice that $\on{Pic}^0(\mathscr{S}_F)$ is a group of multiplicative type, because over $k_{\on{sep}}$, it is isomorphic to a finite product of copies of $\BZ/2\BZ$ and is thus diagonalizable. Consequently, we can apply~\eqref{eq-brcont} to $\on{Pic}^0(\mathscr{S}_F)$-torsors.

A Galois \'{e}tale covering defined over $k$ with Galois group isomorphic to $\on{Pic}^0(\mathscr{S}_F)$ as $G_k$-modules is called a $2$\emph{-covering}. Our strategy for describing the $2$-coverings of $\wt{S}_{F}$ is to describe the $2$-coverings of $\scr{S}_F$ and pull back along the quotient map $\wt{S}_{F} \to [\wt{S}_F/\mathbb{G}_m] = \scr{S}_{F}$. We now briefly recall the results of~\cite[\S5]{MR2156713}, in which Bruin and Flynn explicitly construct all $2$-coverings of $\scr{S}_{F}$ up to isomorphism over $k$ (see also~\cite[\S3]{MR1916903}). Consider the projective space $\BP_k^{2n}$ with homogeneous coordinates $[Z_1 : \cdots : Z_{2n+1}]$ such that $G_k$ acts on $Z_i$ by $\sigma \cdot Z_i = Z_{\sigma(i)}$. The \emph{distinguished covering} is defined to be the closed subscheme $C^1 \subset \BP_k^{2n}$ cut out by the homogeneous ideal
$$I^1 \defeq \big((\theta_i - \theta_j)(Z_\ell^2 - Z_m^2) - (\theta_\ell - \theta_m)(Z_i^2 - Z_j^2) : i,j,\ell,m \in \{1,\dots,2n+1\}\big).$$
Let $\pi^1 \colon C^1 \to \BP_k^1$ be the map defined by
$$\pi^1([Z_1 : \cdots : Z_{2n+1}]) = \frac{\theta_j Z_i^2 - \theta_i Z_j^2}{Z_i^2 - Z_j^2}$$
for any distinct $i,j \in \{1, \dots, 2n+1\}$.  For each class $\delta \in H^1(G_k, \on{Pic}^0(\mathscr{S}_F))$, the twist of \mbox{$\pi^1 \colon C^1 \to \BP_k^1$} by $\delta$ is given explicitly as follows. By~\cite[last paragraphs of \S4.2 and \S5.1]{MR3156850}, $H^1(G_\BQ, \on{Pic}^0(\mathscr{S}_F))$ can be identified with $(K_F^\times/K_F^{\times 2})_{\on{N}\equiv 1}$, so we may think of $\delta$ as being an element of $(K_F^\times/K_F^{\times 2})_{\on{N}\equiv 1}$. Let $(\delta_1, \dots, \delta_{2n+1})$ be the image of $\delta$ under the map $K_F^\times/K_F^{\times 2} \to \big((k_{\on{sep}})^\times /(k_{\on{sep}})^{\times 2}\big)^{2n+1}$ given by taking the product of the distinct embeddings of each field component of $K_F$ into $k_{\on{sep}}$. Then the twist of $C^1$ by $\delta$ is the closed subscheme $C^\delta \subset \BP_k^{2n}$ cut out by the homogeneous ideal
$$I^\delta \defeq \big((\theta_i - \theta_j)(\delta_\ell Z_\ell^2 - \delta_m Z_m^2) - (\theta_\ell -\theta_m)(\delta_i Z_i^2 - \delta_j Z_j^2) : i,j,\ell,m \in \{1, \dots, 2n+1\}\big),$$
and the twist of $\pi^1 \colon C^1 \to \BP_k^1$ is the map $\pi^\delta \colon C^\delta \to \BP_k^1$ defined by
$$\pi^\delta([Z_1 : \cdots : Z_{2n+1}]) = \frac{\theta_j \delta_i Z_i^2 - \theta_i \delta_j Z_j^2}{\delta_iZ_i^2 - \delta_jZ_j^2}$$
for any distinct $i,j \in \{1, \dots, 2n+1\}$. The following theorem enumerates the basic properties of the maps $\pi^\delta \colon C^\delta \to \BP_k^1$:
\begin{theorem}[\protect{\cite[Lemma~5.10]{MR2156713} and~\cite[\S3.1]{MR1916903}}] \label{thm-covprops}
With notation as above, we have:
\begin{enumerate}[leftmargin=2em]
 \item[$\mathrm{(a)}$] For each $\delta \in (K_F^\times/K_F^{\times 2})_{\on{N}\equiv 1}$, the scheme $C_\delta$ is a smooth, geometrically irreducible, projective curve of genus $n \cdot 2^{2n-1} - 3 \cdot 2^{2n-2} + 1$;
 \item[$\mathrm{(b)}$] The action of $G_k$ leaves both $I^\delta$ and $\pi^\delta$ invariant, so the map $\pi^\delta \colon C^\delta \to \BP_k^1$ is defined over $k$;
 \item[$\mathrm{(c)}$] The map $\pi^\delta$ is a covering map with Galois group generated by the automorphisms $\tau_i \colon C^\delta \to C^\delta$ defined by $Z_i \to -Z_i$ for each $i \in \{1, \dots, 2n+1\}$; and
 \item[$\mathrm{(d)}$] We have that $\deg \pi^\delta = 2^{2n}$, that $\# (\pi^{\delta})^{-1}(\theta_i) = 2^{2n-1}$, and that $\#(\pi^\delta)^{-1}(\alpha) = 2^{2n}$ for each $\alpha \in \BP_k^1(k_{\on{sep}}) \smallsetminus \{\theta_1, \dots, \theta_{2n+1}\}$.
 \end{enumerate}
 \end{theorem}
We now demonstrate that the map $\pi^\delta$ can actually be thought of as a $2$-covering $\pi^\delta \colon C^\delta \to \scr{S}_F$.
\begin{proposition}
For any $\delta \in (K_F^\times/K_F^{\times 2})_{\on{N}\equiv 1}$, the map $\pi^\delta$ factors through the coarse moduli map $\scr{S}_F \to \BP_k^1$ to give a $2$-covering $\pi^\delta \colon C^\delta \to \scr{S}_F$.
\end{proposition}
\begin{proof}
Recall that a finite \'{e}tale cover of $\scr{S}_F$ is the pullback via $\scr{S}_F \to \BP_k^1$ of a finite cover of its coarse moduli space $\BP_k^1$ that is ramified with ramification index $2$ at every point in the preimage of $\theta_i \in \BP_k^1(k_{\on{sep}})$ for each $i$ (to account for the fact that $\scr{S}_F$ has a $\frac{1}{2}$-point at each $\theta_i$) and unramified everywhere else. It follows from part (d) of Theorem~\ref{thm-covprops} that the branch locus of the map $\pi^\delta$ consists precisely of the points $\theta_i$ and that each point in $(\pi^{\delta})^{-1}(\theta_i)$ has ramification index $2$. Thus, the map $\pi^\delta$ factors through the stacky curve $\scr{S}_F$, giving an \'{e}tale map $\pi^{\delta} \colon C^{\delta} \to \scr{S}_F$.

By parts (b) and (c) of Theorem~\ref{thm-covprops}, to prove that $\pi^{\delta}$ is a $2$-covering of $\scr{S}_F$ over $k$, it suffices to show that the group generated by the automorphisms $\tau_i$ is isomorphic to $\on{Pic}^0(\mathscr{S}_F)$ as $G_k$-modules. Let $J$ denote the Jacobian of the monic odd-degree hyperelliptic curve $y^2 = F_{\on{mon}}(x,1)$. We claim that the map $\tau_i \mapsto (\theta_i - \infty) \in J[2](k_{\on{sep}})$ defines a $G_k$-equivariant isomorphism of groups. To see why this claim holds, note that the group generated by the automorphisms $\tau_i$ is isomorphic to
$$(\BZ/2\BZ)\langle \tau_1, \dots, \tau_{2n+1}\rangle\bigg/\left(\prod_{i = 1}^{2n+1} \tau_i = 1\right)$$
and the action of $\sigma \in G_k$ on $\tau_i$ is given by $\sigma \cdot \tau_i = \tau_{\sigma^{-1}(i)}$, because the action of $\sigma$ on $Z_i$ is defined to be $\sigma(Z_i) = Z_{\sigma(i)}$. On the other hand, we have that
$$J[2](k_{\on{sep}})= (\BZ/2\BZ)\langle (\theta_1 - \infty), \dots, (\theta_{2n+1} - \infty)\rangle\bigg/\left(\sum_{i = 1}^{2n+1} (\theta_i - \infty)\right)$$
and the action of $\sigma \in G_k$ on $(\theta_i - \infty)$ is $(\theta_i - \infty) \cdot \sigma = (\sigma^{-1} \cdot \theta_i - \sigma^{-1} \cdot \infty) =  (\theta_{\sigma^{-1}(i)} - \infty)$. Thus, we have the claim. The proposition now follows from the fact that $J[2] \simeq \on{Pic}^0(\mathscr{S}_{F})$ as $G_k$-modules (see the proof of Proposition~\ref{prop-torsor}).
\end{proof}
\begin{remark}
Note that two coverings $\pi^{\delta_1} \colon C^{\delta_1} \to \scr{S}_F$ and $\pi^{\delta_2} \colon C^{\delta_2} \to \scr{S}_F$ are isomorphic (as coverings of $\scr{S}_F$ defined over $k$) if and only if $\delta_1$ and $\delta_2$ represent the same class in $(K_F^\times/K_F^{\times 2})_{\on{N}\equiv 1}$.
\end{remark}

Let $(x_0, y_0, z_0) \in \scr{S}_F(k)$, and let $\delta \in K_F^\times$ be the element associated to the point $(x_0, y_0, z_0)$ via the construction in \S\ref{sec-buildabear}. By~\eqref{eq-inspect}, we have that
$$\delta = \begin{cases} x_0 - \theta z_0 & \text{ if $y_0 \neq 0$,} \\ \wt{F}(z_0\theta, z_0) + (x_0 - \theta z_0) & \text{ if $y_0 = 0$} \end{cases}$$
where $\wt{F}$ is as in~\eqref{eq-defftilde}. With this notation, we have the following result, which tells us which $2$-covering of $\scr{S}_F$ has the property that $(x_0, y_0, z_0)$ lies in the image of its $k$-rational points:
\begin{lemma} \label{lem-final}
We have that $(x_0, y_0, z_0) \in \pi^{f_0\cdot \delta}(C^{f_0\cdot \delta}(k)) \subset \scr{S}_F(k)$. 
\end{lemma}
\begin{proof}
First suppose $y_0 \neq 0$. By construction, the point $[1: 1 : \cdots : 1] \in C^{f_0 \cdot \delta}(k)$ satisfies
\begin{align}
\pi^{f_0 \cdot \delta}([1: 1 : \cdots : 1]) & = \frac{\theta_j (f_0 \cdot\delta_i) - \theta_i (f_0 \cdot \delta_j)}{f_0 \cdot \delta_i - f_0 \cdot \delta_j} = \frac{\theta_j \cdot f_0(x_0 - \theta_i z_0) - \theta_i \cdot f_0( x_0 - \theta_j z_0)}{f_0(x_0 - \theta_i z_0) - f_0( x_0 - \theta_j z_0)} \label{eq-sonorepeat}\\
& =  [x_0 : z_0] \in \BP_k^1(k) = \scr{S}_F(k). \nonumber
\end{align}
Now suppose $y_0 = 0$. Then one readily checks that the calculation in~\eqref{eq-sonorepeat} goes through by replacing $[1 : 1 : \cdots : 1]$ with the point $[0 : 1 : 1 : \cdots : 1] \in C^{f_0 \cdot \delta}(k)$.
\end{proof}
We have thus completed our description of the $2$-coverings of $\scr{S}_F$. The $2$-coverings of $\wt{S}_F$ are precisely the pullbacks of the coverings $\pi^\delta \colon C^\delta \to \scr{S}_F$ via the quotient map $\wt{S}_F \to \scr{S}_F$.

\subsection{Proof of Theorem~\ref{thm-hasse}} \label{sec-63}

Let $F \in \mathscr{F}_{2n+1}(f_0)$ be irreducible, let $\delta \in (K_F^\times/K_F^{\times 2})_{\on{N}\equiv 1}$, and consider the associated $2$-covering $\pi^\delta \colon C^\delta \to \scr{S}_F$ over $\BQ$. We say that the $2$-covering $\pi^\delta$ is \emph{locally soluble} if for every place $v$ of $\BQ$, there exists a point $P_v \in \scr{S}_F(\BZ_v)$ such that the number $\delta_v$ associated to $P_v$ via~\eqref{eq-inspect} is such that $f_0 \cdot \delta_v$ is equal to the image of $\delta$ under the natural map $(K_F^\times/K_F^{\times 2})_{\on{N}\equiv1} \to ((K_F\otimes_\BQ \BQ_v)^\times/(K_F\otimes_\BQ \BQ_v)^{\times 2})_{\on{N}\equiv1} $. (Notice that in this situation, $P_v \in \pi^\delta(C^\delta(\BQ_v))$ by Lemma~\ref{lem-final}.) Just as Bruin and Stoll did for hyperelliptic curves in~\cite[\S2]{MR2521292}, we define the \emph{fake $2$-Selmer set} $\on{Sel}_{\on{fake}}^2(\scr{S}_F) \subset (K_F^\times/K_F^{\times 2})_{\on{N}\equiv1}$ to be the set of locally soluble $2$-coverings of $\scr{S}_F$.

We claim that each element $\delta \in \on{Sel}_{\on{fake}}^2(\scr{S}_F)$ gives rise to an orbit of $G(\BZ)$ on $V(\BZ)$, the associated $G(\BQ)$-orbit of which is represented by the class $\delta \in (K_F^\times/K_F^{\times 2})_{\on{N}\equiv 1}$. For each place $v$, let $P_v \in \scr{S}_F(\BZ_v) \cap \pi^\delta(C^\delta(\BQ_v))$. Then via the construction in \S\ref{sec-buildabear}, $P_v$ naturally gives rise to an orbit of $G(\BZ_v)$ on $V(\BZ_v)$, the associated $G(\BQ_v)$-orbit of which is represented by the class $\delta \in ((K_F\otimes_\BQ \BQ_v)^\times/(K_F\otimes_\BQ \BQ_v)^{\times 2})_{\on{N}\equiv1}$. Since the algebraic group $G$ has class number equal to $1$, and since the orbit of $G(\BQ_v)$ on $V(\BQ_v)$ associated to $\delta$ has a representative over $\BZ_v$ for every place $v$, it follows that the orbit of $G(\BQ)$ on $V(\BQ)$ has representative over $\BZ$. Thus, we have the claim.

In the next lemma, we relate the fake $2$-Selmer set to the Brauer--Manin obstruction:

\begin{lemma} \label{lem-selberg}
If $\on{Sel}_{\on{fake}}^2(\scr{S}_F) = \varnothing$, then $\scr{S}_F$ has a Brauer--Manin obstruction to having a $\BZ$-point.
\end{lemma}
\begin{proof}
Let $\pi \colon C \to \wt{S}_{F,\BQ}$ be any $2$-covering. We have by~\eqref{eq-needlabel1},~\eqref{eq-brdef},~\eqref{eq-needlabel2}, and~\eqref{eq-brcont} that
\begin{align*}
\scr{S}_F(\BA_\BQ)^{\on{Br}} & \subset \wt{S}_F(\BA_\BQ)^{\on{int}} \cap \wt{S}_{F,\BQ}(\BA_\BQ)^\pi = \left(\prod_{v} \scr{S}_F(\BZ_v)\right) \cap \left(\bigcup_{\delta \in H^1(G_\BQ, \on{Pic}^0(\scr{S}_F)(\ol{\BQ}))} \pi^\delta(C^\delta(\BA_\BQ)) \right) \\
&  \subset \bigcup_{\delta \in \on{Sel}_{\on{fake}}^2(\scr{S}_F)} \pi^\delta(C^\delta(\BA_\BQ)),
\end{align*}
and clearly the last union above is empty if $\on{Sel}_{\on{fake}}^2(\scr{S}_F) = \varnothing$.
\end{proof}
In proving Theorem~\ref{thm-main}, we imposed local conditions on $V(\BZ)$ to sieve to the orbits that arise from local points at each place. By definition, the exact same local conditions define orbits that arise from elements of fake $2$-Selmer sets. Moreover, if $F \in \mathscr{F}_{2n+1}(f_0) \smallsetminus \mathscr{F}_{2n+1}^*(f_0)$, Theorem~\ref{thm-caseprime} implies that the distinguished covering is not everywhere locally soluble. Thus, by the argument in \S\ref{sec-finally}, we have the following result:
\begin{theorem} \label{thm-notlast}
The upper density of forms $F \in \mathscr{F}_{2n+1}(f_0)$ such that $\on{Sel}_{\on{fake}}^2(\scr{S}_F) \neq \varnothing$ is at most $\mu_{f_0} + o(2^{-n})$ when $2 \nmid f_0$ and at most $\mu_{f_0} + O(2^{-\varepsilon_1 n^{\varepsilon_2}})$ for some $\varepsilon_1, \varepsilon_2 > 0$ when $2 \mid f_0$.
\end{theorem}
To prove Theorem~\ref{thm-hasse}, it remains to determine how often it is that the stacky curves $\scr{S}_F$ have integral points everywhere locally. We do so as follows:
\begin{lemma} \label{lem-neededanother}
The density of forms $F \in \mathscr{F}_{2n+1}(f_0)$ such that $\scr{S}_F(\BZ_v) \neq \varnothing$ for every place $v$ is at least $\mu_{f_0}' + O(2^{-2n})$.
\end{lemma}
\begin{proof}
By~\eqref{eq-triv}, we always have $\mathscr{S}_F(\BR) \neq \varnothing$. If there is no prime $p \mid \upkappa$ such that $F(x_0,z_0) = 0$ for every $[x_0 : z_0] \in \BP^1(\BZ/p\BZ)$, then for any prime $p \mid \upkappa$, there exists a pair $(x_0, z_0) \in \BZ^2$ such that $\gcd(x_0, z_0) = \gcd(F(x_0, z_0),p) = 1$, so $(x_0 \cdot F(x_0, z_0), F(x_0, z_0)^{n+1}, z_0 \cdot F(x_0, z_0)) \in \mathscr{S}_F(\BZ_p)$. On the other hand, if $p \mid \upkappa$ and the mod-$p$ reduction of $F$ has a simple zero at some point of $\BP^1(\BZ/p\BZ)$, then $\mathscr{S}_F$ has a Weierstrass point over $\BZ_p$ and is thus soluble. Moreover, for any prime $p$ such that $p \nmid \upkappa$ but $p \mid f_0$, we have $(\upkappa,f_0^{\frac{1}{2}}\upkappa^{\frac{2n+1}{2}},0) \in \mathscr{S}_F(\BZ_p)$, and for any prime $p \nmid f_0$, we have $(f_0, f_0^{n+1},0) \in \scr{S}_F(\BZ_p)$.

If $p \mid \upkappa$, the $p$-adic density of forms $F \in \mathscr{F}_{2n+1}(f_0)$ such that the mod-$p$ reduction of $F$ has a zero of multiplicity greater than $1$ at each point of $\BP^1(\BZ/p\BZ)$ is $p^{-2p-1}$ if $2p+2 < 2n+1$ and $p^{-2n-1}$ otherwise. Thus, the density of forms $F \in \mathscr{F}_{2n+1}(f_0)$ such that $\scr{S}_F(\BZ_v) \neq \varnothing$ for every place $v$ is at least $\prod_{p \mid \upkappa} (1 - p^{-2p-1} - p^{-2n-1}) = \mu_{f_0}' + O(2^{-2n})$.
\end{proof}

It follows from Theorem~\ref{thm-notlast} and Lemma~\ref{lem-neededanother} that for all sufficiently large $n$, a positive proportion of $F \in \mathscr{F}_{2n+1}(f_0)$ are such that $\on{Sel}_{\on{fake}}^2(\scr{S}_F) = \varnothing$ and $\scr{S}_F(\BZ_v) \neq \varnothing$ for every place $v$. Indeed,
\begin{align*}
(1 - \mu_{f_0}) + (\mu_{f_0}' + O(2^{-2n})) & = 1 - \prod_{p \mid \upkappa} \frac{1}{p^2} + \prod_{p \mid \upkappa} \left(1 - \frac{1}{p^{2p+1}}\right) + O(2^{-n}),
\end{align*}
 which is clearly greater than $1$ for every sufficiently large $n$. Theorem~\ref{thm-hasse} now follows from Lemma~\ref{lem-selberg}.

\begin{acknowledgements}
    It is a pleasure to thank Manjul Bhargava for suggesting the questions that led to this paper and for providing invaluable advice and encouragement. We are immensely grateful to Nils Bruin, Sungmun Cho, Benedict Gross, Bjorn Poonen, Peter Sarnak, Arul Shankar, Michael Stoll, David Zureick-Brown, and Jerry Wang for answering our questions and for sharing their insights. We thank the anonymous referee for providing numerous helpful comments and suggesting several revisions that have improved the readability of the paper. We also thank Levent Alp\"{o}ge, Joe Harris, Aaron Landesman, Daniel Loughran, Anand Patel, Beth Romano, Efthymios Sofos, James Tao, and Melanie Wood for helpful discussions.
  \end{acknowledgements}

	\bibliographystyle{amsalpha}
	\bibliography{bibfile}

\end{document}